\documentclass[12pt]{article}
\usepackage{amsmath,amssymb,amsthm}
\usepackage{graphicx,enumerate,geometry,graphicx,epstopdf}
\usepackage{dcolumn,latexsym}
\usepackage{a4wide,pdfsync,epstopdf,color}
\usepackage{fancybox}
\usepackage[graph]{xy}
\usepackage[colorlinks=true,linkcolor=blue,citecolor=blue]{hyperref}
\usepackage{array}
\usepackage{rotating}
\usepackage{caption}
\usepackage{subcaption}
\usepackage{booktabs}
\usepackage{dcolumn}
\usepackage{marginnote}
\usepackage[x11names,usenames,dvipsnames,svgnames,table]{xcolor}
\numberwithin{equation}{section}
\newtheorem{theorem}{Theorem}[section]

\newtheorem{corollary}[theorem]{Corollary}
\newtheorem{proposition}[theorem]{Proposition}
\newtheorem{lemma}[theorem]{Lemma}
\theoremstyle{definition}
\newtheorem{algorithm}{Algorithm}[section]
\newtheorem{subroutine}[algorithm]{Subroutine}
\newtheorem{remark}{Remark}[section]
\newtheorem{definition}{Definition}[section]

\newtheorem{assumption}{Assumption}
\newcommand{\ie}{{\em i.e.}, }
\newcommand{\eg}{{\em e.g.}, }

\newcommand{\etal}{{\em et al.\ }}
\newcommand{\nn}{\mathbb{N}} %Nonnegative integers
\newcommand{\real}{\mathbb{R}} %Real numbers
\newcommand{\comp}{\mathbb{C}} %Complex numbers
\newcommand{\norm}[1]{\left\Vert {#1} \right\Vert} %Norm
\newcommand{\erl}{\left(-\infty , +\infty\right]} %Extended real line
\newcommand{\dom}[1]{\mathrm{dom}\,{#1}} %Domain
\DeclareMathOperator*{\argmin}{arg\,min}

\newcommand{\dist}{\mathrm{dist}} %Distance between point and set
\newcommand{\act}[1]{\left\langle {#1} \right\rangle} %The value of 1

\newcommand{\CCC}{\mathcal{C}}
\newcommand{\FFF}{\mathcal{F}} % Fourier Transform
\newcommand{\MMM}{\mathcal{M}}
\newcommand{\Ibb}{\mathbb{I}}
\newcommand{\iF}{\mathcal{F}^{-1}} % Fourier Transform
\newcommand{\bz}{{\bf z}}
\title{Proximal Heterogeneous Block Input-Output Method and application to Blind Ptychographic Diffraction Imaging}
\author{Robert Hesse\thanks{Institut f\"ur Numerische und Angewandte Mathematik,
Universit\"at G\"ottingen,\ Lotzestr.~16--18, 37083 G\"ottingen, Germany. E-mail: \texttt{hesse@math.uni-goettingen.de}.} \and 
D.\ Russell Luke\thanks{Institut f\"ur Numerische und Angewandte Mathematik,
Universit\"at G\"ottingen,\ Lotzestr.~16--18, 37083 G\"ottingen, Germany. E-mail: \texttt{r.luke@math.uni-goettingen.de}.} \and 
Shoham Sabach\thanks{Institut f\"ur Numerische und Angewandte Mathematik,
Universit\"at G\"ottingen,\ Lotzestr.~16--18, 37083 G\"ottingen, Germany. E-mail: \texttt{s.sabach@math.uni-goettingen.de}.}
\and 
Matthew K. Tam\thanks{CARMA Centre, University of Newcastle, Callaghan, NSW 2308, Australia. 
E-mail: \texttt{matthew.tam@uon.edu.au}.}}
\date{\today}

\begin{document}
\maketitle

\begin{abstract}
	We propose a general alternating minimization algorithm for nonconvex optimization 
	problems with separable structure and nonconvex coupling between blocks of variables. To 
	fix our ideas, we apply the methodology to the problem of blind ptychographic imaging.
	Compared to other schemes in the literature, our approach differs in two ways: (i) it is 
	posed within a clear mathematical framework with practically verifiable assumptions, and 
	(ii) under the given assumptions, it is provably convergent to critical points. A 
	numerical comparison of our proposed algorithm with the current state-of-the-art on 
	simulated and experimental data validates our approach and points toward directions for 
	further improvement. 
	
\end{abstract}

\section{Introduction}
	We consider algorithms for nonconvex constrained optimization problems of the following 
	form	
	\begin{equation} \label{LeastSquaresFormulation}
		\mbox{ Find } \left(\overline{x} , \overline{y} , \overline{z}\right) \in 
		\argmin\left\{ F\left(x , y , z\right)~|~\left(x , y , z\right) \in C \equiv X \times 
		Y \times Z \right\}.
	\end{equation}
	Here $X \times Y \times Z \subset \real^{p} \times \real^{q} \times \real^{r}$ (that is,  
	the constraints apply to disjoint blocks of variables) and $F$ is a nonlinear penalty 
	function characterizing the coupling between the blocks of variables. It will be 
	convenient to reformulate problem \eqref{LeastSquaresFormulation} using indicator 
	functions. The indicator function of a set $C$ is defined as $\iota_{C}\left(x\right) = 
	0$ for $x \in C$ and $\iota_{C}\left(x\right) = +\infty$ for $x \notin C$. Define 
	\begin{equation} \label{e:Psi}
		\Psi\left(x , y , z\right) \equiv F\left(x , y , z\right) + \iota_{X}\left(x\right) + 
		\iota_{Y}\left(y\right) + \iota_{Z}\left(z\right).   
	\end{equation}
	An equivalent formulation of \eqref{LeastSquaresFormulation} is the formally 
	unconstrained nonsmooth optimization problem
	\begin{equation} \label{LeastSquaresFormulation-Psi}
		\mbox{ Find } \left(\overline{x} , \overline{y} , \overline{z}\right) \in 
		\argmin_{\left(x , y , z\right) \in \real^{p} \times \real^{q} \times \real^{r}} 
		\left\{ \Psi\left(x , y , z\right) \right\}.
	\end{equation}
	Algorithms for solving \eqref{LeastSquaresFormulation} or 
	\eqref{LeastSquaresFormulation-Psi} typically seek only to satisfy first-order necessary 
	conditions for optimality, and the algorithm we propose below is no different. These 
	conditions are given compactly by 
	\begin{equation}\label{e:Fermat}
   		0 \in \nabla F\left(x^{*} , y^{*} , z^{*}\right) + 
   		\partial\iota_{X}\left(x^{*}\right) + \partial\iota_{Y}\left(y^{*}\right) + 
   		\partial\iota_{Z}\left(z^{*}\right),
	\end{equation}
	where $\partial f\left(z\right)$ is a set, the {\em subdifferential}, that generalizes 
	the notion of a gradient for nonsmooth, subdifferentially regular functions $f$ defined 
	precisely in Definition \ref{d:subdifferential} below.  
\medskip

	For the sake of fixing the ideas, we focus on the particular application of {\em blind 
	ptychography}, however our goal and approach are much more general. The partially smooth 
	character of the objective $\Psi$ in \eqref{e:Psi} is a common feature in many 
	optimization models which involve sums of functions, some of which are smooth and some of 
	which are not. Forward-backward-type algorithms are frequently applied to such models, 
	and our approach is no different. The particular three-block structure of the problem is 
	easily generalized to $M$ blocks. The crucial feature of the model for algorithms, and 
	what we hope to highlight in the present study, is the quantification of continuity of 
	the partial gradients of $F$ with respect only to blocks of variables. This is in 
	contrast to more classical approaches which rely on the continuity of $\nabla F$ with 
	respect to {\em all the variables} simultaneously (see \cite{ABS2013}). For the 
	ptychography application, such a requirement prohibits a convergence analysis along the 
	lines of \cite{ABS2013} since the gradient $\nabla F$ is {\em not} Lipschitz continuous. 
	However, the partial gradients with respect to the blocks of variables {\em are} 
	Lipschitz continuous. Following \cite{BST2013}, this allows us to prove, in 
	Section~\ref{Sec:Analysis}, convergence of the blocked algorithm given below 
	(Algorithm~\ref{a:PBIE}) to feasible critical points.  
\medskip

	Our abstract formulation of the blind ptychography problem can be applied to many 
	different applications including control, machine learning, and deconvolution. We do not 
	attempt to provide a review of the many different approaches to these types of problems, 
	or even a more focused review of numerical methods for ptychography, but rather to 
	provide a common theoretical framework by which a variety of methods can be understood. 
	Our focus on ptychography is due to the success of two algorithms, one by Maiden and 
	Rodenburg \cite{MR2009} and the other due to Thibault and collaborators \cite{TDBMP2009}.  
	These two touchstone methods represent, for us, fundamental computational methods whose 
	structure serves as a central bifurcation in numerical strategies.  Moreover, the 
	prevalence of these two methods in practice ensures that our theoretical framework will 
	have the greatest practical impact.  (Which is not to say that the methods are the most  
	efficient \cite{YangMarchesini, SECCMS}.) We present an algorithmic framework in Section 
	\ref{Sec:formulation} by which these algorithms can be understood and analyzed. We 
	present in Section~\ref{Sec:Analysis} a theory of convergence of the most general 
	Algorithm \ref{a:PBIE} which is refined with increasingly stringent assumptions until it 
	achieves the form of Algorithm~\ref{a:PPHeBIE} that can be immediately applied to 
	ptychography. The specialization of our algorithm to ptychography is presented in Section 
	\ref{Sec:AppPyt} and summarized in Algorithm~\ref{a:PPHeBIE ptych}. We compare, in 
	Section \ref{Sec:Implementing}, Algorithm~\ref{a:PPHeBIE ptych} with the state-of-the-art 
	on simulated and experimental data.

\section{Algorithms and Modeling}\label{Sec:formulation}
	The solution we seek is a triple, $(\overline{x},\overline{y},\overline{z})$ that 
	satisfies \emph{a priori} constraints, denoted by $C$, as well as a model characterizing 
	the coupling between the variables. We begin naively with a very intuitive idea for 
	solving \eqref{LeastSquaresFormulation}: alternating minimization (AM) with respect to 
	the three separate blocks of variables $x$, $y$ and $z$. More precisely, starting with 
	any $\left(x^{0} , y^{0} , z^{0}\right) \in X \times Y \times Z$, we consider the 
	following algorithm:
	\begin{subequations} \label{a:AM}
   		\begin{align}
			x^{k + 1} & \in \argmin_{x \in X} \left\{ F\left(x , y^{k} , z^{k}\right) 
			\right\}, \label{e:AM1} \\
			y^{k + 1} & \in \argmin_{y \in Y} \left\{ F\left(x^{k + 1} , y , z^{k}\right) 
			\right\}, \label{e:AM2} \\
			z^{k + 1} & \in \argmin_{z \in Z} \left\{ F\left(x^{k + 1} , y^{k + 1} , z\right) 
			\right\}. \label{e:AM3}
		\end{align}
	\end{subequations}
	\noindent While the simplicity of the above algorithm is attractive, there are several 
	considerations one must address:
	\begin{itemize}
    	\item[$\rm{(i)}$] The convergence results for the AM method are limited and 
    		applicable only in the convex setting. It is unknown if the AM method converges 
    		in the nonconvex setting. Of course, in the general nonconvex setting we can not 
    		expect for convergence to global optimum but even convergence to critical points 
    		is not known. In \cite{ABRS2010} the authors prove convergence to critical points 
    		for a regularized variant of alternating minimization.  We follow this approach 
    		in Algorithm~\ref{a:PBIE} below, applying proximal regularization in each of the 
    		steps to obtain provable convergence results. 
    	\item[$\rm{(ii)}$] Each one of the steps of the algorithm involves solving an 
    		optimization problem over just one of the blocks of variables. 
    		Forward-backward-type methods are common approaches to solving such minimization 
    		problems \cite{CW2005}, and can be mixed between blocks of 
    		variables \cite{BST2013}. Forward operators are typically applied to the 
    		ill-posed or otherwise computationally difficult parts of the objective (usually 
    		appearing within the smooth part of the objective) while backward operators are 
    		applied to the well-posed parts of the objective (appearing often in the 
    		nonsmooth part).  
	\end{itemize}
	In the particular case of ptychography, the subproblems with respect to the $x$ and $y$ 
	variables (steps \eqref{e:AM1} and \eqref{e:AM2} of the algorithm) are ill-posed. We 
	handle this by applying a regularized backward operator (the prox operator) to 
    blockwise-linearizations of the ill-posed forward steps. The objective is well-posed, and 
    particularly simple, with respect to the third block of variables, $z$, so we need only 
    employ a backward operator for this step. Generalizing, our approach addresses issue (ii) 
    above by handling each of the blocks of variables differently. 
\medskip

	Our presentation of Algorithm~\ref{a:PBIE} makes use of the following notation. For any 
	fixed $y \in \real^{q} $ and $z \in \real^{r}$, the function $x \mapsto F\left(x , y , 
	z\right)$ is continuously differentiable and its partial gradient, $\nabla_{x} F\left(x , 
	y , z\right)$, is Lipschitz continuous with moduli $L_{x}\left(y , z\right)$. The same 
	assumption holds for the function $y \mapsto F\left(x , y , z\right)$ when $x \in 
	\real^{p}$ and $z \in \real^{r}$ are fixed. In this case, the Lipschitz moduli is denoted 
	by $L_{y}\left(x , z\right)$. Define $L_{x}'\left(y , z\right) \equiv \max \left\{ 
	L_{x}\left(y , z\right) , \eta_{x} \right\}$ where $\eta_{x}$ is an arbitrary positive 
	number. Similarly define $L_{y}'\left(x , z\right) \equiv \max \left\{ L_{y}\left(x , 
	z\right) , \eta_{y} \right\}$ where $\eta_{y}$ is an arbitrary positive number.  
\vspace{0.2in}

    \fcolorbox{black}{Ivory2}{\parbox{17cm}{
	\begin{algorithm}[{\bf Proximal Block Implicit-Explicit Algorithm}] \label{a:PBIE}$~$\\
    {\bf Initialization.} Choose $\alpha , \beta > 1$, $\gamma > 0$ and   
     $\left(x^{0} , y^{0} , z^{0}\right) \in X \times Y \times Z$. \\
    {\bf General Step ($k = 0 , 1 , \ldots$)}
	\begin{itemize}
		\item[1.] Set $\alpha^{k} = \alpha L_{x}'\left(y^{k} , z^{k}\right)$ and select
			\begin{equation} \label{Algo:Step1}
				x^{k + 1} \in \argmin_{x \in X} \left\{ \act{x - x^{k} , \nabla_{x} 
				F\left(x^{k} , y^{k} , z^{k}\right)} + \frac{\alpha^{k}}{2}\norm{x - 
				x^{k}}^{2} \right\}, 
			\end{equation}
		\item[2.] Set $\beta^{k} = \beta L_{y}'\left(x^{k  + 1} , z^{k}\right)$ and select
			\begin{equation} \label{Algo:Step2}
				y^{k + 1} \in \argmin_{y \in Y} \left\{ \act{y - y^{k} , \nabla_{y} 
				F\left(x^{k + 1} , y^{k} , z^{k}\right)} + \frac{\beta^{k}}{2}\norm{y - 
				y^{k}}^{2} \right\},
			\end{equation}
		\item[3.] Select
			\begin{equation} \label{Algo:Step3}
				z^{k + 1} \in \argmin_{z \in Z} \left\{ F\left(x^{k + 1} , y^{k + 1} , 
				z\right) + \frac{\gamma}{2}\norm{z - z^{k}}^{2} \right\}.
			\end{equation}
	\end{itemize}
	\end{algorithm}}}
\vspace{0.2in}

	The regularization parameters $\alpha^{k}$ and $\beta^{k}$, $k \in \nn$, are discussed in 
	Section~\ref{Sec:Implementing}. For the moment, suffice it to say that these parameters 
	are inversely proportional to the stepsize in Steps~\eqref{Algo:Step1} and 
	\eqref{Algo:Step2} of the algorithm (see Section~\ref{Sec:Analysis}). Noting that 
	$\alpha_{k}$ and $\beta_{k}$, $k \in \nn$, are directly proportional to the respective 
	partial Lipschitz moduli, the larger the partial Lipschitz moduli the {\em smaller} the 
	stepsize, and hence the slower the algorithm progresses.  
\medskip

	This brings to light another advantage of blocking strategies that goes beyond 
	convergence proofs: algorithms that exploit block structures inherent in the objective 
	function achieve better numerical performance by taking heterogeneous step sizes 
	optimized for the separate blocks. There is, however, a price to be paid in the blocking 
	strategies that we explore here: namely, they result in procedures that pass {\em 
	sequentially} between operations on the blocks, and as such are not immediately 
	parallelizable. Here too, the ptychography application generously rewards us with added 
	structure, as we show in Section~\ref{SSec:Acceleration}, permitting parallel 
	computations on highly segmented blocks.      
\medskip

	The convergence theory developed in Section~\ref{Sec:Analysis} is independent of the 
	precise form of the coupling function $F$ and independent of the precise form of the 
	constraints. For our analysis we require that $F$ is differentiable with $\nabla F$ 
	Lipschitz continuous on bounded domains, and the partial gradient $\nabla_{x} F$  
	(globally) Lipschitz continuous as a mapping on $X$ for each $\left(y , z\right) \in Y 
	\times Z$ fixed, and partial gradient $\nabla_{y} F$ (globally) Lipschitz continuous as a 
	mapping on $Y$ for $\left(x , z\right) \in X \times Z$ fixed. The constraint sets $X$, 
	$Y$ and $Z$ could be very general, we only assume that they are closed and disjoint. This 
	is discussed more precisely below. The analysis presented in Section~\ref{Sec:Analysis} 
	guarantees only that Algorithm~\ref{a:PBIE} converges to a point satisfying 
	\eqref{e:Fermat}, which, it is worth reiterating, are not necessarily solutions to 
	\eqref{LeastSquaresFormulation}.  

\subsection{Blind Ptychography} \label{SSec:Pyt}
	In scanning ptychography, an unknown specimen is illuminated by a localized 
	electromagnetic beam and the resulting wave is recorded on a CCD array somewhere 
	downstream along the axis of propagation of the wave (\ie in the {\em far field} or the 
	{\em near field} of the object). A ptychographic dataset consists of a series of such 
	observations, each differing from the others by a spatial shift of either the object or 
	the illumination. In the original ptychographic reconstruction procedure \cite{Hoppe} it 
	was assumed that the illuminating beam was known. What we call {\em blind} ptychography, 
	in analogy with blind deconvolution, reflects the fact that the beam is not completely 
	known, this corresponds to what is commonly understood by ptychography in modern 
	applications \cite{RodenburgBates, Rodenburg07, MR2009, TDBMP2009}. Here the problem is 
	to simultaneously reconstruct the specimen and illuminating beam from a given 
	ptychgraphic dataset. We will treat the case of scanning x-ray ptychography with far 
	field measurements. This is not exhaustive of all the different settings one might 
	encounter, but the mathematical structure of the problem, our principal interest, is 
	qualitatively the same for all cases. For a review of ptychothographic methods, see 
	\cite{R2008} and the reference herein.
\medskip
 
 	We formulate the ptychography problem on the product space $\comp^{n} \times \comp^{n} 
	\times \comp^{n \times m}$ where the first block $\comp^{n}$ corresponds to the model 
	space for the probe, the second block corresponds to the model space for the specimen, 
	and the third block corresponds to the model space for the data/observations. The 
	physical model space equipped with the real inner product is isomorphic to the Euclidean 
	space $\left(\real^{2}\right)^{n} \times \left(\real^{2}\right)^{n} \times 
	\left(\real^{2}\right)^{n \times m}$ with the inner product $\act{\left(x , y , 
	z\right) , \left(x' , y', z'\right)} \equiv \sum_{j = 1}^{n} \act{x_{j} , x'_{j}} + 
	\sum_{j = 1}^{n} \act{y_{j} , y'_{j}} + \sum_{j = 1}^{n}\sum_{i = 1}^m \act{z_{ij} , 
	z'_{ij}}$ for $x_{i} , y_{i} , z_{ij} \in \real^{2}$.  This is in fact how complex 
	numbers are represented on a computer, and hence the model space $\comp^{n}$ with real 
	inner product is just an efficient shorthand for $\left(\real^{2}\right)^{n}$ with the 
	standard inner product for such product spaces.  We will therefore retain the complex 
	model space with real inner product when describing this problem, noting that all linear 
	operators on this space have analogues on the space $\left(\real^{2}\right)^{n}$.  The 
	theory, however, will be set on real finite dimensional vector spaces.  
\medskip

	Denote $\bz \equiv \left(z_{1} , z_{2} , \ldots , z_{m}\right)$ with $z_{j} \in 
	\comp^{n}$ ($j = 1 , 2 , \ldots , m$).  The objective function in our general 
	optimization problem \eqref{LeastSquaresFormulation}, $F : \comp^{n} \times \comp^{n} 
	\times \comp^{n \times m} \rightarrow \real_{+}$, is given by 
	\begin{equation} \label{e:F}
		F\left(x , y , \bz\right) \equiv \sum_{j = 1}^{m} \norm{S_{j}\left(x\right) \odot y - 
		z_{j}}^{2}.
	\end{equation}
	Here $S_{j} : \comp^{n} \to \comp^{n}$ denotes $j$-th shift operator which shifts the 
	indexes in $x \in \comp^{n}$ in some prescribed fashion and $\odot$ is the elementwise 
	Haadamard product. This function measures in some sense the distance to the set 
	\begin{equation}\label{EquationFormulation}
   		\MMM \equiv \left\{ \left(x , y , \bz\right) \in \comp^{n} \times \comp^{n} \times 
   		\comp^{n \times m} ~|~ S_{j}\left(x\right) \odot y = z_{j}, \,\, j = 1 , 2 , \ldots 
   		, m\right\}.
	\end{equation}

	Let $b_{j} \in \real^{n}$, $j = 1 , 2 , \ldots , m$, denote the experimental 
	observations, and $\FFF$ be a $2$D discrete Fourier transform (of dimension $\sqrt{n} 
	\times \sqrt{n}$) rearranged for vectors on $\comp^{n}$. The constraints $X$, $Y$, and 
	$Z$ are separable and given by 
	\begin{subequations}\label{e:C}
 		\begin{align} \label{XConstraint} 
 			X & \equiv \{\mbox{qualitative constraints on the probe}\}, \\
			Y & \equiv \{\mbox{qualitative constraints on the specimen}\}, \\
			Z & \equiv Z_{1} \times Z_{2} \times \cdots \times Z_{m}, \nonumber \\
			  & \mbox{where } Z_{j} \equiv \left\{ z \in \comp^{n}~|~\left| \left(\FFF 
			  z\right)_{l} \right| = b_{jl}, \; (l = 1 , 2 , \ldots , n) \right\} \quad (j = 
			  1 , 2 , \ldots , m). \label{e:Mj}
		\end{align}
	\end{subequations}
	The qualitative constraints characterized by $X$ and $Y$ are typically a mixture of 
	support, support-nonnegativity or magnitude constraints corresponding respectively to 
	whether the illumination and specimen are supported on a bounded set, whether these (most 
	likely only the specimen) are ``real objects'' that somehow absorb or attenuate the probe 
	energy, or whether these are ``phase'' objects with a prescribed intensity but varying 
	phase. A support constraint for the set $X$, for instance, would be represented by 
	\begin{equation} \label{e:support}
		X \equiv \left\{ x = (x_1,x_2,\dots,x_n) \in \comp^{n}~|~|x_{i}| \leq R~ (i=1,2,
		\dots,n)\mbox{ and, for } i \notin \Ibb_{X},~ x_{i} = 0 \right\},
	\end{equation}
	where $\Ibb_{X}$ is the index set corresponding to which pixels in the field of view the 
	probe beam illuminates and $R$ is some given amplitude. A mixture of support and 
	amplitude constraints for the set $Y$ 
	would be represented by 
	\begin{equation} \label{e:support-amplitude}
		Y \equiv \left\{ y = (y_1,y_2,\dots,y_n) \in \comp^{n}~|~0 \leq \underline{\eta} \leq \left|y_{i}\right| \leq 
		\overline{\eta}\mbox{ and,  for } i \notin \Ibb_{Y},  y_{i} = 0 \right\},
% 		\cap \left\{ y = (y_1,y_2,\dots,y_n)\in \comp^{n}~|~0 \leq \underline{\eta} \leq \left|y_{i}\right| \leq 
% 		\overline{\eta}~ (i=1,2,\dots,n) \right\},
	\end{equation}
	where the index set $\Ibb_{Y}$ is the analogous index set for the support of the 
	specimen, and $\underline{\eta}/\overline{\eta}$ are lower/upper bounds on the intensity 
	of the specimen. The set $Z$ is nothing more than the {\em phase set} appearing in 
	feasibility formulations of the phase retrieval problem \cite{LBL2002}.  
	\begin{remark}[Feasibility versus minmization]
		Since the algorithms discussed below involve, at some point, projections onto these 
		sets, it is worthwhile noting here that, while, in most applications, the projections 
		onto the sets $X$, $Y$ and $Z$ have a closed form and can be computed very accurately 
		and efficiently, we are unaware of any method, analytic or otherwise, for computing 
		the projection onto the set $\MMM$ defined by \eqref{EquationFormulation}. For this 
		reason, we have avoided formulation of the problem as a (nonconvex) {\em feasibility} 
		problem 
		\begin{equation*}
   			\mbox{Find }\overline{x} \in \MMM \cap \left(X \times Y \times Z \right). 
		\end{equation*}
		Nevertheless, this essentially two-set feasibility model suggests a wide range of 
		techniques within the family of projection methods, alternating projections, averaged 
		projections and Douglas--Rachford being representative members. In contrast to these, 
		our approach is essentially a {\em forward-backward} method that avoids the 
		difficulty of computing a projection onto the set $\MMM$ by instead minimizing a 
		nonnegative coupling function $F$ that takes the value $0$ (only) on $\MMM$. 
	\end{remark}

\section{Algorithm Analysis} \label{Sec:Analysis}
\subsection{Mathematical Preliminaries} \label{SSec:Preliminaries}
	Algorithm~\ref{a:PBIE} consists of three steps all of which reduce to the computation of 
	a \emph{projection} onto a given constraint set (convex and nonconvex). Recall that the 
	\emph{projection} onto a nonempty and closed subset $\Omega$ of a Euclidean space 
	$\real^{d}$, is the (set-valued) mapping $P_{\Omega} : \real^{d} \rightrightarrows 
	\Omega$ defined by
    \begin{equation} \label{D:OrthogonalProjection}
        P_{\Omega}\left(v\right) \equiv \argmin \left\{ \norm{u - v}~|~u \in \Omega \right\}.
    \end{equation}
    In Euclidean spaces the projection is single valued if and only if $\Omega$ is convex (in 
    addition to being nonempty and closed). Specializing to the present application, for the 
    constraints specified by \eqref{e:C}, $P_{C} = \left(P_{X} , P_{Y} , P_{Z}\right)$ where 
    $P_{X}$, $P_{Y}$ and $P_{Z}$ are, in general, multivalued (consider 
    $P_{C}\left(0\right)$). 
\medskip

	Since we are dealing with nonsmooth and nonconvex functions that can take the value 
	$+\infty$, we require the following generalization of the derivative for nonconvex 
	functions.
	\begin{definition}[Subdifferential \cite{RW1998-B}]\label{d:subdifferential}
		Let $f : \real^{d} \rightarrow \erl$ be proper (not everywhere infinite) and lower 
		semicontinuous (lsc).    
		\begin{itemize}
   			\item The {\em regular or Fr\'echet subdifferential} of $f$ at $u \in 
   				\dom f$, denoted $\widehat{\partial}f(u)$, is the set of vectors $v \in 
   				\real^{d}$ which satisfy
				\begin{equation} \label{e:rsd}
   					\liminf_{\stackrel{w \neq u}{w \rightarrow u}}\frac{f\left(w\right) - 
   					f\left(u\right) - \act{v , w - u}}{\norm{w - u}} \geq 0.
				\end{equation}
				If $u \notin \dom f$ then $\widehat{\partial}f\left(u\right) \equiv 
				\emptyset$. 
			\item The {\em limiting subdifferential} of $f$ at $u \in \dom f$, 
				denoted $\partial f\left(u\right)$, is the set of limits of limiting 
				subdifferentials: 
				\begin{equation*} 
					\partial f\left(u\right) \equiv \left\{ v \in \real^{d} ~|~ \exists u^{k} 
					\rightarrow u \mbox{ with } f\left(u^{k}\right) \rightarrow 									
					 f\left(u\right) \mbox{ and } v^{k} \rightarrow v \mbox{ with } v^{k} \in 
					\widehat{\partial} f\left(u^{k}\right) \mbox{ as } k \rightarrow \infty 
					\right\}.
				\end{equation*}
		\end{itemize}
		We say that $f$ is {\em subdifferentially regular} at $u$ if $\partial 		
		f\left(u\right) = \widehat{\partial} f\left(u\right)$, and {\em subdifferentially 
		regular} (without reference to the point $u$) if it is subdifferentially regular at 
		every point in $\dom f$.   
	\end{definition}
    The notion of regularity of a set can be understood in terms of the subdifferential 
    regularity of the indicator function of that set. We will call a set {\em Clarke regular} 
    if the corresponding indicator function is subdifferentially regular (see 
    \cite[Definition 6.4]{RW1998-B}). 
\medskip
    
    In \cite{BST2013}, Bolte \etal present a general procedure for determining convergence to 
    critical points of generic algorithms for nonsmooth and nonconvex problems. The procedure 
    consists of verifying three criteria, two of which are quite standard and shared by most 
    \textit{descent} algorithms, see \eg \cite{ABS2013}. The third criterion depends not on 
    the algorithm but on the objective function: it must satisfy the {\em 
    Kurdyka-{\L}ojasiewicz (KL) property} (see \cite{BDL2006, BDL2007} and the references 
    therein). 
	\begin{definition}[Kurdyka-{\L}ojasiewicz property]\label{d:KL}
		Let $f : \real^{d} \rightarrow \erl$ be proper and lower semicontinuous. For $\eta 
		\in \left(0 , +\infty\right]$ define
		\begin{equation}
   			\CCC_{\eta} \equiv \left\{ \varphi \in C\left[\left[0 , \eta\right) ,
   			\real_{+}\right] \mbox{ such that } 
   			\left.\begin{cases}
			   	\varphi\left(0\right) = 0 & \\ 
			   	\varphi \in C^{1} & \mbox{ on } \left(0 ,\eta\right) \\
				\varphi'\left(s\right) > 0 & \mbox{ for all } s \in \left(0 , \eta\right)
			\end{cases} \right\} \right\}.
		\end{equation}
		The function $f$ is said to have the Kurdyka-{\L}ojasiewicz (KL) property at 
		$\overline{u} \in \dom \partial f$ if there exist $\eta \in \left(0 , 
		+\infty\right]$, a neighborhood $U$ of $\overline{u}$ and a function $\varphi \in 
		\CCC_{\eta}$, such that, for all 
		\begin{equation*}
		 u \in U \cap [f(\overline{u}) < f(u) < f(\overline{u}) + \eta],
		\end{equation*}
		the following inequality holds
		\begin{equation}\label{e:KL}
   			\varphi\left(f\left(u\right) - f\left(\overline{u}\right)\right)\dist\left(0 , 
   			\partial f\left(u\right)\right) \geq 1.
		\end{equation}
		If $f$ satisfies property~\eqref{e:KL} at each point of $\dom \partial f$, then 
		$f$ is called a {\em KL function}.
	\end{definition}
	For a given function, the KL property can be verified indirectly by checking membership 
	to certain classes of functions, in particular the class of {\em semi-algebraic} 
	functions \cite{BDL2006}. For the convenience of the reader, we recall here the 
	definition of semi-algebraic functions.
    \begin{definition}[Semi-algebraic sets and functions]\hspace*{\fill}
        \begin{itemize}
            \item[$\rm{(i)}$] A $S\subseteq\real^{d}$ is \emph{(real) semi-algebraic} if 
            	there exists a finite number of real polynomial functions $p_{ij} , q_{ij} : 
            	\real^{d} \rightarrow \real$ such that
                \begin{equation*}
                    S = \bigcup_{j = 1}^{N} \bigcap_{i = 1}^{K} \left\{ u \in \real^{d} : 
                    \; p_{ij}\left(u\right) = 0 \, \text{ and } \, q_{ij}\left(u\right) < 
                    0 \right\}.
                \end{equation*}
            \item[$\rm{(ii)}$] A function $f : \real^{d} \rightarrow \erl$ is \emph{semi-
            	algebraic} if its graph
                \begin{equation*}
                    \left\{ \left(u , t\right) \in \real^{d + 1}~|~f\left(u\right) = t 
                    \right\},
                \end{equation*}
                is a semi-algebraic subset of $\real^{d + 1}$.
        \end{itemize}
    \end{definition}
    The class of semi-algebraic sets is stable under the following operations: finite unions, 
    finite intersections, complementation and Cartesian products. For a thorough catalog of 
    semi-algebraic functions and sets see \cite{AB2009, ABRS2010, ABS2013, BST2013} and the 
    references therein.
\medskip

	While it may not be obvious how {\em directly} to verify the KL property it is easy to 
	determine whether a function is semi-algebraic. The remarkable fact about semi-algebraic 
	functions is that, as long as they are lower semi-continuous, they automatically satisfy 
	the KL property on their domain as stated in the following result. 
	\begin{theorem}[{\cite[Th.~3.3, pg.~1215]{BDL2006}}] \label{P:Semi-aKL}
        Let $f : \real^{d} \rightarrow \erl$ be a proper lower semicontinuous function. 
        If $f$ is semi-algebraic then it satisfies the KL property at any point in $\dom{f}$.
    \end{theorem}
	We now show that the ptychography problem, and hence the phase retrieval problem, is 
	semi-algebraic, \ie both the objective function $F$ and the constraint set $C$ are 
	semi-algebraic.
	\begin{proposition}[Blind ptychography and phase retrieval are semi-algebraic] 
	\label{t:Ptychography semialg}
		The objective function $F$, defined by \eqref{e:F}, is continuous and semi-algebraic. 
		The constraint sets $X$, $Y$ and $Z$, defined by \eqref{e:C}, are nonempty, closed 
		and semi-algebraic. Consequently, the corresponding function $\Psi$, defined by 
		\eqref{e:Psi}, is a KL function on $X\times Y\times Z$.
	\end{proposition}
	\begin{proof}[Proof sketch]
		The physical model is formulated with complex-valued vectors, but $\comp^{n}$ with 
		the real inner product is isomorphic to the Euclidean space 
		$\left(\real^{2}\right)^{n}$ with the inner product $\act{x , x'} \equiv \sum_{i = 
		1}^{n} \left(x_{i} , x'_{i}\right)$ for $x_{i} , x'_{i} \in \real^{2}$. The function 
		$F$ defined by \eqref{e:F} is finite everywhere, continuous (indeed, differentiable), 
		and the level sets of the objective $F$ are quadratics with respect to $y$ and 
		$z_{j}$, and quadratic with respect to $x$ under linear transformations. Thus $F$ is 
		semi-algebraic. The sets $X$ and $Y$ are either subspaces (support constraint only, 
		\eqref{e:support}) or the intersection of a subspace with a box or ball (support and 
		nonnegativity or support and amplitude constraints \eqref{e:support-amplitude}), and 
		so both of these are nonempty semi-algebraic. The set $Z$ is equivalent to an 
		amplitude constraint in the image space of the linear mapping $\FFF$ with respect to 
		the 1-norm on each two-dimensional component of the product space 
		$\left(\real^{2}\right)^{n}$. Thus $Z$ is also nonempty semi-algebraic.  That $\Psi$ 
		defined by \eqref{e:Psi} is then a KL-function for these $F$, $X$, $Y$, and $Z$ then 
		follows from Theorem \ref{P:Semi-aKL}.
	\end{proof}

\subsection{Convergence Analysis}\label{Sec:Convergence}
	Our convergence analysis is centered on Theorem~\ref{T:Convergence}, a general result
    concerning the application of Algorithm~\ref{a:PBIE} to problem \eqref{e:Fermat}. The 
    specialization to the ptychography problem, Proposition \ref{t:Ptychography hypotheses}, 
    is then easily achieved by verifying that the assumptions of a refinement, Theorem 
    \ref{T:Convergence gen parallel}, are satisfied. Following \cite{BST2013}, we carry out 
    the three-step procedure, outlined in Section~\ref{SSec:Preliminaries}, for proving 
	convergence of the sequence $\left\{\left(x^{k} , y^{k} , z^{k}\right) \right\}_{k \in 
	\nn}$, generated by Algorithm~\ref{a:PBIE}, to a point satisfying \eqref{e:Fermat} 
	provided that the initial point $\left(x^{0} , y^{0} , z^{0}\right) \in X \times Y \times 
	Z$. The analysis rests on the following assumptions, collected here to avoid repetition.
\vspace{0.2in}

    \fcolorbox{black}{white}{\parbox{17cm}{
		\begin{assumption} \label{hyp} 
			Let $\left\{ \left(x^{k} , y^{k} , z^{k}\right) \right\}_{k \in \nn}$ be iterates 
			of Algorithm~\ref{a:PBIE} for  $\left(x^{0} , y^{0} , z^{0}\right) \in X \times Y 
			\times Z$.
			\begin{enumerate}[(i)]
				\item \label{hyp:1} $X \subset \real^{p}$, $Y \subset \real^{q}$, and 
					$Z \subset \real^{r}$ are nonempty and closed.  
				\item \label{hyp:2} $F : \real^{p} \times \real^{q} \times \real^{r} 
					\rightarrow \real$ is differentiable on $X \times Y \times Z$ and 
					$\inf F> -\infty$. Moreover, $\nabla_{x} F$ and $\nabla_{y} F$ 
					(as defined above) are Lipschitz continuous with moduli 
					$L_{x}\left(y , z\right)$ and $L_{y}\left(x , z\right)$, respectively. 
				\item \label{hyp:3} The gradient of $F$, $\nabla F$, is Lipschitz continuous 
					on bounded domains in $X\times Y\times Z$. Moreover, there exists 
					$\lambda_{x}^{+} , \lambda_{y}^{+} > 0$ such that
					\begin{equation} \label{e:upper}
						\sup \left\{ L_{x}\left(y^{k} , z^{k}\right)~|~k \in \nn \right\} 
						\leq \lambda_{x}^{+} \quad \mbox{ and } \quad \sup \left\{ 				
						L_{y}\left(x^{k + 1} , z^{k}\right)~|~k \in \nn\right\} \leq 
						\lambda_{y}^{+}.
					\end{equation}
				\item \label{hyp:4} The iterates $\left\{\left(x^{k}, y^{k}, z^{k}\right) 
					\right\}_{k \in \nn}$ are bounded. 
				\item \label{hyp:5} The function $\Psi$ defined by \eqref{e:Psi} is a KL 
					function (see Definition \ref{d:KL}).
			\end{enumerate}
		\end{assumption}}}
\vspace{0.2in}	

	\begin{remark}\label{r:ptychography assumptions}
		In Section~\ref{Sec:AppPyt}, we will show that our ptychography model (described in 
		Section~\ref{SSec:Pyt}) satisfies these assumptions. For the general setting  we 
		point out the following. 
		\begin{enumerate}[(i)]
			\item\label{r:compact constraints}	Boundedness of the iterates 
				Assumption~\ref{hyp}\eqref{hyp:4}) is a strong assumption that can be    
				handled by a more technical treatment than we would like to present here. For 
				our purposes this can be guaranteed by the physically natural assumption that 
				the constraint set is bounded. Since Algorithm~\ref{a:PBIE} is a {\em 
				feasible point} algorithm, all iterates belong to the bounded feasible set, 
				hence, in this case, the iterates are bounded.
			\item Combining Assumptions~\ref{hyp}\eqref{hyp:2} and \eqref{hyp:3} do not 
				guarantee that the gradient $\nabla F$ is globally Lipschitz, as is the case 
				in the application described below (see Section \ref{Sec:AppPyt}). The 
				inequalities in \eqref{e:upper} could be obtained in several scenarios, for 
				example, when $F$ is $C^{2}$ and using the boundedness assumption 
				Assumption~\ref{hyp}\eqref{hyp:4}.
		\end{enumerate}
	\end{remark}
	We begin with a technical lemma.
    \begin{lemma}[Sufficient decrease property] 
    \label{L:Sufficent}
        Let $h : \real^{d} \rightarrow \real$ be a continuously differentiable function with 
        gradient $\nabla h$ assumed to be $L_{h}$-Lipschitz continuous and let $\Omega$ be a 
        nonempty and closed subset of $\real^{d}$. Fix any $t > L_{h}$. Then, for any $u \in 
        \Omega$ and for $u^{+} \in \real^{d}$ defined by
        \begin{equation*}
            u^{+} \in P_{\Omega}\left(u - \frac{1}{t}\nabla h\left(u\right)\right),
        \end{equation*}
        we have
        \begin{equation*}
            h\left(u^{+}\right) \leq h\left(u\right) - \frac{1}{2}\left(t - L_{h}\right) 
            \norm{u^{+} - u}^{2}.
        \end{equation*}
    \end{lemma}
    \begin{proof}
    	The result follows from \cite[Lemma~2]{BST2013} where the nonsmooth 
    	function $\sigma$ is the indicator function $\iota_{\Omega}$ of the nonempty 
    	closed set $\Omega$.
    \end{proof}
    \begin{remark}
    	When $\Omega$ is also convex, the conclusion of Lemma~\ref{L:Sufficent} 
    	can be improved (see \cite[Lemma~2.3]{BT2009}) to the following
        \begin{equation*}
            h\left(u^{+}\right) \leq h\left(u\right) - \left(t - \frac{L_{h}}{2}\right) 
            \norm{u^{+} - u}^{2}.
        \end{equation*}
    This means that $t > L_{h}/2$ (rather than only $t > L_{h}$, as in the nonconvex case) is 
    enough to guarantee decrease of function value after projected-gradient step.
    \end{remark}
	Using Lemma~\ref{L:Sufficent} we can prove the following basic property of 
	Algorithm~\ref{a:PBIE}.
	\begin{proposition}[Sufficient decrease] \label{P:FiniteLength2}
		Let $\left\{ \left(x^{k} , y^{k} , z^{k}\right) \right\}_{k \in \nn}$ be a sequence 
		generated by Algorithm~\ref{a:PBIE} for some initial point $\left(x^{0} , y^{0} , 
		z^{0}\right) \in X \times Y \times Z$. Suppose that conditions 		
		\eqref{hyp:1}-\eqref{hyp:2} of Assumption~\ref{hyp} hold. Then the sequence 
		$\left\{ F\left(x^{k} , y^{k} , z^{k}\right) \right\}_{k \in \nn}$ is decreasing and 
		\begin{equation*}
			\sum_{k = 1}^{\infty} \norm{\left(x^{k + 1} , y^{k + 1} , z^{k + 1}\right) - 
			\left(x^{k} , y^{k} , z^{k}\right)}^{2} < \infty.
		\end{equation*}
		Hence the sequence $\left\{ F\left(x^{k} , y^{k} , z^{k}\right) \right\}_{k \in \nn}$ 
		converges to some $F^{\ast} > -\infty$ as $k \rightarrow \infty$. 
	\end{proposition}
	\begin{proof}
		We apply Lemma \ref{L:Sufficent} to the first subproblem (see \eqref{Algo:Step1}) as 
		follows. Take $h\left(\cdot\right) = F\left(\cdot , y^{k} , z^{k}\right)$, $\Omega = 
		X$ and $t = \alpha^{k} > L_{x}'\left(y^{k} , z^{k}\right)$ to obtain that 
		\begin{align*}
			F\left(x^{k + 1} , y^{k} , z^{k}\right) & \leq F\left(x^{k} , y^{k} , z^{k} 
			\right) - \frac{1}{2}\left(\alpha^{k} - L_{x}\left(y^{k} , z^{k}\right)\right) 
			\norm{x^{k + 1} - x^{k}}^{2} \\
			& \leq F\left(x^{k} , y^{k} , z^{k} \right) - \frac{1}{2}\left(\alpha^{k} - 
			L_{x}'\left(y^{k} , z^{k}\right)\right)\norm{x^{k + 1} - x^{k}}^{2} \\
			& = F\left(x^{k} , y^{k} , z^{k} \right) - \frac{1}{2}\left(\alpha - 1\right) 
			L_{x}'\left(y^{k} , z^{k}\right)\norm{x^{k + 1} - x^{k}}^{2} \\
			& \leq F\left(x^{k} , y^{k} , z^{k} \right) - \frac{1}{2}\left(\alpha - 1\right) 
			\eta_{x}\norm{x^{k + 1} - x^{k}}^{2},	
		\end{align*}
		where the second inequality follows from the fact that $L_{x}\left(y^{k} , 
		z^{k}\right) \leq L_{x}'\left(y^{k} , z^{k}\right)$ and the last inequality follows 
		from the fact that $\eta_{x} \leq L_{x}'\left(y^{k} , z^{k}\right)$ and $\alpha > 1$. 
		Similarly, applying Lemma \ref{L:Sufficent} to the second subproblem (see 
		\eqref{Algo:Step2}) with $h\left(\cdot\right) = F\left(x^{k + 1} , \cdot , 
		z^{k}\right)$, $\Omega = Y$ and $t = \beta^{k} > L_{y}'\left(x^{k + 1} , 
		z^{k}\right)$ yields
		\begin{equation*}
			F\left(x^{k + 1} , y^{k + 1} , z^{k}\right) \leq F\left(x^{k + 1} , y^{k} , 
			z^{k}\right) - \frac{1}{2}\left(\beta - 1\right)\eta_{y}\norm{y^{k + 1} - 
			y^{k}}^{2}.
		\end{equation*}
		On the other hand, immediately from the third updating rule (see \eqref{Algo:Step3}) 
		we get that
		\begin{equation*}
			F\left(x^{k + 1} , y^{k + 1} , z^{k + 1}\right) \leq F\left(x^{k + 1} , y^{k + 1} 
			, z^{k}\right) - \frac{\gamma}{2}\norm{z^{k + 1} - z^{k}}^{2}.
		\end{equation*}
		Summing up all these inequalities yields
		\begin{align*}
			F\left(x^{k + 1} , y^{k + 1} , z^{k + 1}\right) & \leq F\left(x^{k} , y^{k} , 
			z^{k}\right) - \frac{1}{2}\left(\alpha - 1\right)\eta_{x}\norm{x^{k + 1} - 
			x^{k}}^{2} - \frac{1}{2}\left(\beta - 1\right)\eta_{y}\norm{y^{k + 1} - 
			y^{k}}^{2} \\
			& - \frac{\gamma}{2}\norm{z^{k + 1} - z^{k}}^{2}.
		\end{align*}
		Denote $\lambda^{-} \equiv \left(1/2\right)\min \left\{ \left(\alpha - 1\right) 
		\eta_{x} , \left(\beta - 1\right)\eta_{y} , \gamma \right\}$. Thus
		\begin{equation}
			F\left(x^{k + 1} , y^{k + 1} , z^{k + 1}\right) \leq F\left(x^{k} , y^{k} , 
			z^{k}\right) - \lambda^{-}\norm{\left(x^{k + 1} , y^{k + 1} , z^{k + 1}\right) - 
			\left(x^{k} , y^{k} , z^{k}\right)}^{2}. 		
		\end{equation}
		This proves that the sequence $\left\{ F\left(x^{k} , y^{k} , z^{k}\right) 
		\right\}_{k \in \nn}$ is decreasing. Since, in addition, we know that $F$ is bounded 
		from below (see Assumption \ref{hyp}\eqref{hyp:2}), we thus have a decreasing 
		sequence on a compact interval and it follows that $\left\{ F\left(x^{k} , y^{k} , 
		z^{k}\right) \right\}_{k \in \nn}$ converges to some $F^{\ast} > -\infty$. Summing up 
		this inequality, for $k = 1 , 2 , \ldots, N$, yields
		\begin{align} 
			\sum_{k = 1}^{N} \norm{\left(x^{k + 1} , y^{k + 1} , z^{k + 1}\right) - 
			\left(x^{k} , y^{k} , z^{k}\right)}^{2} & \leq \frac{1}{\lambda^{-}} 
			\left(F\left(x^{1} , y^{1} , z^{1}\right) - F\left(x^{N + 1} , y^{N + 1} , z^{N + 
			1}\right)\right) \nonumber\\
			& \leq \frac{F\left(x^{1} , y^{1} , z^{1}\right) - F^{\ast}}{\lambda^{-}},
			\label{P:FiniteLength2:1}
		\end{align}
		where the last inequality holds true since $F\left(x^{N + 1} , y^{N + 1} , z^{N + 
		1}\right) \geq F^{\ast}$. Taking the limit as $N \rightarrow \infty$ yields 
		boundedness of the sum of step-lengths and completes the proof.
	\end{proof}
	Before proving the second step, we obtain the following immediate consequence.
	\begin{corollary}[Rate of asymptotic regularity]
		Let $\left\{ \left(x^{k} , y^{k} , z^{k}\right) \right\}_{k \in \nn}$ be a sequence 
		generated by Algorithm~\ref{a:PBIE} for some initial point $\left(x^{0} , y^{0} , 
		z^{0}\right) \in X \times Y \times Z$ and define the corresponding sequence of steps 
		$\left\{ s^{k} \right\}_{k \in \nn \setminus\{0\}}$ by $s^{k + 1} \equiv \left(x^{k 
		+ 1} , y^{k + 1} , z^{k + 1}\right) - \left(x^{k} , y^{k} , z^{k}\right)$. Suppose 
		that conditions \eqref{hyp:1}-\eqref{hyp:2} of Assumption~\ref{hyp} hold. Then $s^{k} 
		\rightarrow 0$ as $k \rightarrow \infty$ with the following rate
   		\begin{equation*}
    			\min_{k = 1 , 2 , \ldots , N} \norm{s^{k + 1}} \leq \sqrt{\frac{F\left(x^{1} 
    			, y^{1} , z^{1}\right) - F^{\ast}}{N\lambda^{-}}},
    		\end{equation*}
    	where $\lambda^{-} \equiv \left(1/2\right)\min \left\{ \left(\alpha - 1\right) 
    	\eta_{x} , \left(\beta - 1\right)\eta_{y} , \gamma \right\}$ and $F^{\ast} \equiv 
    	\lim_{k \rightarrow \infty} F\left(x^{k} , y^{k} , z^{k}\right)$.
	\end{corollary}
	\begin{proof}
		From \eqref{P:FiniteLength2:1} we obtain that
		\begin{equation*}
			N\min_{k = 1 , 2 , \ldots , N} \norm{s^{k + 1}}^{2} \leq \sum_{k = 1}^{N} 
			\norm{s^{k + 1}}^{2} \leq \frac{F\left(x^{1} , y^{1} , z^{1}\right) - 
			F^{\ast}}{\lambda^{-}},
		\end{equation*}
		and thus
		\begin{equation*}
			\min_{k = 1 , 2 , \ldots , N} \norm{s^{k + 1}}^{2} \leq \frac{F\left(x^{1} , 
			y^{1} , z^{1}\right) - F^{\ast}}{N\lambda^{-}}.
		\end{equation*}
		The result now easily follows. 
	\end{proof}	
    \begin{proposition}[Lipschitz paths]\label{t:bounded}
		Let $\left\{ (x^{k} , y^{k} , z^{k}) \right\}_{k \in \nn}$ be a sequence generated by 
		Algorithm \ref{a:PBIE} for some initial point $\left(x^{0} , y^{0} , z^{0}\right) \in 
		X \times Y \times Z$. Suppose that conditions \eqref{hyp:1}-\eqref{hyp:4} of 
		Assumption \ref{hyp} hold. For each positive integer $k$, define the following three 
		quantities: $A_{z}^{k} \equiv \gamma\left(z^{k - 1} - z^{k}\right)$,
        \begin{equation*}
            A_{x}^{k} \equiv \alpha^{k - 1}\left(x^{k - 1} - x^{k}\right) + \nabla_{x} 
            F\left(x^{k} , y^{k} , z^{k}\right) - \nabla_{x} F\left(x^{k - 1} , y^{k - 1} , 
            z^{k - 1}\right),
        \end{equation*}  
        and
        \begin{equation*}
            A_{y}^{k} \equiv \beta^{k - 1}\left(y^{k - 1} - y^{k}\right) + \nabla_{y} 
            F\left(x^{k} , y^{k} , z^{k}\right) - \nabla_{y} F\left(x^{k}  , y^{k - 1} , z^{k 
            - 1}\right).
        \end{equation*}  
        Then $A^{k} \equiv \left(A_{x}^{k} , A_{y}^{k} , A_{z}^{k}\right) \in \partial 
        \Psi\left(x^{k} , y^{k} , z^{k}\right)$ and there exists $\delta > 0$ such that
        \begin{equation*}
        	\norm{A^{k}} \leq \left(3\lambda^{+} + 2\delta\right)\norm{\left(x^{k} , y^{k} , 
        	z^{k}\right) - \left(x^{k - 1} , y^{k - 1} , z^{k - 1} \right)},
        \end{equation*}
        where $\lambda^{+} \equiv \max \left\{ \lambda_{x}^{+} , \lambda_{y}^{+} , \gamma 
        \right\}$
    \end{proposition}
    \begin{proof}
    	Let $k$ be a positive integer. Writing the optimality condition of the first 
    	updating rule yields
        \begin{equation*}
            \nabla_{x} F\left(x^{k - 1} , y^{k - 1} , z^{k - 1}\right) + \alpha^{k - 
            1}\left(x^{k} - x^{k - 1}\right) + w_{x}^{k}  = 0,
        \end{equation*}
        where $w_{x}^{k} \in \partial \iota_{X}\left(x^{k}\right)$. Hence
        \begin{equation*}
            \nabla_{x} F\left(x^{k - 1} , y^{k - 1} , z^{k - 1}\right) + w_{x}^{k} = 
            \alpha^{k - 1}\left(x^{k - 1} - x^{k}\right).
        \end{equation*}
        It is clear from the definition of $\Psi$ (see \eqref{e:Psi}), that
        \begin{equation*}
    		\partial_{x} \Psi\left(x^{k} , y^{k} , z^{k}\right) = \nabla_{x} F\left(x^{k} , 
    		y^{k} , z^{k}\right) + \partial\iota_{X}\left(x^{k}\right).
        \end{equation*}
        Combining these two facts proves that $A_{x}^{k} \in \partial_{x} \Psi\left(x^{k} , 
        y^{k}, z^{k}\right)$. Following the same arguments applied on the second updating 
        rule yields the desired result that $A_{y}^{k} \in \partial_{y} \Psi\left(x^{k} , 
        y^{k}, z^{k}\right)$. Now, writing the optimality condition of the third updating 
        rule yields 
    	\begin{equation*}
    		\nabla_{z} F\left(x^{k}, y^{k}, z^{k}\right) + \gamma\left(z^{k} - z^{k - 
    		1}\right) + w_{z}^{k} = 0,
    	\end{equation*}
    	for $w_{z}^{k} \in \partial\iota_{Z}\left(z^{k}\right)$, hence $A_{z}^{k} \in 
    	\partial_{z} \Psi\left(x^{k}, y^{k}, z^{k}\right)$.

        We begin with an estimation of the norm of $A_{x}^{k}$. From Assumption~
        \ref{hyp}\eqref{hyp:3} and \eqref{hyp:4}, there exists $\delta > 0$ such that 
        \begin{align*}
    		\norm{\nabla_{x} F\left(x^{k} , y^{k} , z^{k}\right) - \nabla_{x} F\left(x^{k - 
    		1} , y^{k - 1} , z^{k - 1} \right)} & \leq \norm{\nabla F\left(x^{k} , y^{k} , 
    		z^{k}\right) - \nabla F\left(x^{k - 1} , y^{k - 1} , z^{k - 1} \right)} \\
    		& \leq \delta\norm{\left(x^{k} , y^{k} , z^{k}\right) - \left(x^{k - 1} , y^{k - 
    		1} , z^{k - 1} \right)}.
        \end{align*}
        Thus, from the definition of $\lambda^{+}$ we obtain
        \begin{align*}
            \norm{A_{x}^{k}} & \leq \alpha^{k - 1}\norm{x^{k - 1} - x^{k}} + \norm{\nabla_{x} 
            F\left(x^{k} , y^{k} , z^{k}\right) - \nabla_{x} F\left(x^{k - 1} , y^{k - 1} , 
            z^{k - 1}\right)} \\
            & \leq \lambda_{x}^{+}\norm{x^{k - 1} - x^{k}} + \delta\norm{\left(x^{k} , y^{k} 
            , z^{k}\right) - \left(x^{k - 1} , y^{k - 1} , z^{k - 1} \right)} \\
            & \leq \left(\lambda_{x}^{+} + \delta\right)\norm{\left(x^{k} , y^{k} , 
            z^{k}\right) - \left(x^{k - 1} , y^{k - 1} , z^{k - 1} \right)},
        \end{align*}  
        where the second inequality follows from Assumption~\ref{hyp}\eqref{hyp:3}. A similar 
        argument yields
        \begin{equation*}
            \norm{A_{y}^{k}} \leq \left(\lambda_{y}^{+} + \delta\right)\norm{\left(x^{k} , 
            y^{k} , z^{k}\right) - \left(x^{k - 1} , y^{k - 1} , z^{k - 1} \right)}.        
        \end{equation*}
        Thus
        \begin{align*}
            \norm{A^{k}} & \leq \norm{A_{x}^{k}} + \norm{A_{y}^{k}} + \norm{A_{z}^{k}} \leq 
            \left(\lambda_{x}^{+} + \lambda_{y}^{+} + 2\delta\right)\norm{\left(x^{k} , y^{k} 
            , z^{k}\right) - \left(x^{k - 1} , y^{k - 1} , z^{k - 1} \right)} + \gamma 
            \norm{z^{k - 1} - z^{k}} \\
            & \leq \left(3\lambda^{+} + 2\delta\right)\norm{\left(x^{k} , y^{k} , 
            z^{k}\right) - \left(x^{k - 1} , y^{k - 1} , z^{k - 1} \right)}.
        \end{align*}
        This proves the desired result.
    \end{proof}	
	We are now ready to prove the main result of this section, namely convergence of
	Algorithm~\ref{a:PBIE} to points satisfying \eqref{e:Fermat} for any initial point 
	$\left(x^{0} , y^{0} , z^{0}\right) \in X \times Y \times Z$. It is in deducing the last 
	step of the general case that we use the assumption that $\Psi$ satisfies the KL 
	inequality \eqref{e:KL}.
    \begin{theorem}[Convergence to critical points] \label{T:Convergence}
		Let $\left\{ \left(x^{k} , y^{k} , z^{k}\right) \right\}_{k \in \nn}$ be a sequence 
		generated by Algorithm~\ref{a:PBIE} for some initial point $\left(x^{0} , y^{0} , 
		z^{0}\right) \in X \times Y \times Z$. Suppose that Assumption~\ref{hyp} holds. Then 
		following assertions hold.
        \begin{enumerate}[(a)]
            \item The sequence $\left\{ \left(x^{k} , y^{k} , z^{k}\right) \right\}_{k \in 
            	\nn}$ has finite length, that is,
                \begin{equation*}
                    \sum_{k = 1}^{\infty} \norm{\left(x^{k + 1} , y^{k + 1} , z^{k + 
                    1}\right)- \left(x^{k} , y^{k} , z^{k}\right)} < \infty.
                \end{equation*}
            \item The sequence $\left\{ \left(x^{k} , y^{k} , z^{k}\right) \right\}_{k \in 
            	\nn}$ converges to a point $\left(x^{\ast} , y^{\ast} , z^{\ast}\right)$ 
            	satisfying \eqref{e:Fermat}.
        \end{enumerate}
    \end{theorem}
    \begin{proof}
		The result follows from Propositions~\ref{P:FiniteLength2} and \ref{t:bounded} 
		together with \cite[Theorem~1]{BST2013}.
    \end{proof}   

\subsection{Acceleration of the PFB method} \label{SSec:Acceleration}
	In this section we develop an accelerated version of Algorithm~\ref{a:PBIE}. To motivate 
	our approach, we return to the naive alternating minimization method \eqref{a:AM} with 
	which we began. If each of the blocks were themselves separable, then we could 
	recursively apply the blocking strategy discussed in Section~\ref{Sec:formulation} within 
	the blocks. We first detail recursive blocking, which improves the step sizes, and then 
	we discuss additional structures that enable efficient implementations via 
	parallelization. 
\medskip

	For simplicity, we focus our discussion on the first block $X$, the same strategy also 
	can (and will) be applied to the block $Y$. Suppose the block $X$ can be further 
	subdivided into a product of smaller blocks: $X = X_{1} \times X_{2} \times \cdots \times 
	X_{P}$ with $P \leq p$. For fixed $y \in \real^{q}$ and $z \in \real^{r}$ we consider the 
	problem \eqref{e:AM1}, 
	\begin{equation} \label{e:AM11}
	   \min_{x \in X_{1} \times X_{2} \times \cdots \times X_{P}} \left\{ F\left(x , y , 
	   z\right) \right\}.
	\end{equation}
	This problem has the same difficulties with respect to the sub-blocks $X_{i}$, $i = 1 , 2 
	, \ldots , P$, and the other variables $Y$ and $Z$ as the original problem 
	\eqref{LeastSquaresFormulation-Psi} has between the blocks $X$, $Y$ and $Z$. We therefore 
	use the same forward-backward strategy to solve the problem on the block, that is, we 
	partially linearize $F$ with respect to the {\em sub-blocks} of $X$ (as opposed to a 
	partial linearization with respect to the whole block) and compute the corresponding 
	proximal operator.  
\medskip

	More precisely, for $\xi \in X_{i}$ define $\zeta_{i}^{k}\left(\xi\right) \equiv 
	\left(x_{1}^{k + 1} , x_{2}^{k + 1} , \ldots , x_{i - 1}^{k + 1} , \xi, x_{i + 1}^{k} , 
	\ldots , x_{P}^{k}\right) \in X_{1} \times X_{2} \times \cdots \times X_{P}$ and 
	$u_{i}^{k} \equiv \left(x_{1}^{k + 1} , x_{2}^{k + 1} , \ldots , x_{i - 1}^{k + 1} , x_{i 
	+ 1}^{k} , \ldots , x_{P}^{k}\right)$. Let $L_{x_{i}}\left(u_{i}^{k} , y^{k} , 
	z^{k}\right)$ denote the modulus of Lipschitz continuity of the gradient of the mapping 
	$x_{i} \mapsto F\left(\zeta_{i}^{k}\left(x_{i}\right) , y^{k} , z^{k}\right)$. For some 
	$\eta_{x_{i}} > 0$ ($i = 1 , 2 , \ldots , P$) fixed, define $L_{x_{i}}'\left(u_{i}^{k} , 
	y^{k} , z^{k}\right) \equiv \max\left\{ L_{x_{i}}\left(u_{i}^{k} , y^{k} , z^{k}\right) , 
	\eta_{x_{i}} \right\}$. From the iterate $x^{k} = \left(x_{1}^{k} , x_{2}^{k} , 
	\ldots , x_{P}^{k}\right)$ we compute $x^{k + 1} = \left(x_{1}^{k + 1} , x_{2}^{k + 1} , 
	\ldots , x_{P}^{k + 1}\right)$ by the following procedure. 
\vspace{0.2in}

	\fcolorbox{black}{Ivory2}{\parbox{17cm}{
		\begin{subroutine}[Successive sub-block $x$ updating rule] \label{a:subroutine x}
			Define $x_{0}^{k + 1} \equiv x_{1}^{k}$. Given $x_{1}^{k + 1} , x_{2}^{k + 1} , 
			\ldots , x_{i - 1}^{k + 1}$ ($i = 1, 2 , \ldots , P$) compute $x_{i}^{k + 1}$ by 
			\begin{equation*}  
				x_{i}^{k + 1} \in \argmin_{x_{i} \in X_{i}} \left\{ \left\langle\left(x_{i} - 
				x_{i}^{k}\right),  \nabla_{x_{i}} F\left(\zeta_{i}^{k}\left(x_{i}^{k}\right) 
				, y^{k} , z^{k}\right)\right\rangle + \frac{\alpha_{i}^{k}}{2}\left\|x_{i} - 
				x_{i}^{k}\right\|^{2} \right\}, 
			\end{equation*}	
			where $\alpha_{i}^{k} \equiv \alpha_{i}L_{x_{i}}'\left(u_{i}^{k} , y^{k} , 	
			z^{k}\right)$ for some fixed $\alpha_{i} > 1$. 
		\end{subroutine}}}
\vspace{0.2in}

	Comparing this to \eqref{Algo:Step1}, we note that the update for $x^{k + 1}$ computed by 
	Subroutine \ref{a:subroutine x} is computed with different stepsize in {\em each sub-
	block}, where the stepsize $\alpha_{i}^{k}$ depends, again, on the modulus of Lipschitz 
	continuity of the gradient of the function defined on that sub-block. In contrast, the 
	stepsize without recursive blocking, that is the stepsize $\alpha^{k}$ computed according 
	to \eqref{Algo:Step1}, depends on the modulus of Lipschitz continuity of the gradient of 
	the function defined on the entire block, which is, by definition, larger than the 
	constant associated with each sub-block. Consequently, the steps in Algorithm 
	\eqref{a:PBIE} without subblocking will be {\em smaller} than the steps computed via 
	Subroutine \ref{a:subroutine x}.  
\medskip

	Now, repeating this argument for the $Y$-block of variables yields an analogous 
	sequential updating rule for this block. For $\mu \in Y_{j}$ and $Q \leq q$, define 
	$\phi_{j}^{k}\left(\mu\right) \equiv \left(y_{1}^{k + 1} , y_{2}^{k + 1} , \ldots , y_{j 
	- 1}^{k + 1} , \mu, y_{j + 1}^{k} , \ldots , y_{Q}^{k}\right) \in Y_{1} \times Y_{2} 
	\times \cdots \times Y_{M} = Y$ and $v_{j}^{k} \equiv \left(y_{1}^{k + 1} , y_{2}^{k + 1} 
	, \ldots , y_{j - 1}^{k + 1} , y_{j + 1}^{k} , \ldots , y_{Q}^{k}\right)$. Let 
	$L_{y_{j}}\left(x^{k + 1} , v_{j}^{k} , z^{k}\right)$ be the modulus of Lipschitz 
	continuity of the gradient of the function $y_{j} \mapsto F\left(x^{k + 1} , 
	\phi_{j}^{k}\left(y_{j}\right) , z^{k}\right)$. For some $\eta_{y_{j}} > 0$ ($j = 1 , 2 ,
	\ldots , Q$) fixed, define $L_{y_{j}}'\left(x^{k + 1} , v_{j}^{k} , z^{k}\right) \equiv 
	\max\left\{ L_{y_{j}}\left(x^{k + 1} , v_{j}^{k} , z^{k}\right) , \eta_{y_{j}} \right\}$.  
	From the iterate $y^{k} = \left(y_{1}^{k} , y_{2}^{k} , \ldots , y_{Q}^{k}\right)$ we 
	compute $y^{k + 1} = \left(y_{1}^{k + 1} , y_{2}^{k + 1} , \ldots , y_{Q}^{k + 1}\right)$ 
	by the following procedure.
\vspace{0.2in}

	\fcolorbox{black}{Ivory2}{\parbox{17cm}{
		\begin{subroutine}[Successive sub-block $y$ updating rule] \label{a:subroutine y}
 			Define $y_{0}^{k + 1} \equiv y_{1}^{k}$. Given $y_{1}^{k + 1} , y_{2}^{k + 1} , 
 			\ldots , y_{j - 1}^{k + 1}$ ($j = 1 , 2 , \ldots , Q$) compute $y_{j}^{k + 1}$ by 
			\begin{equation*}  
				y_{j}^{k + 1} \in \argmin_{y_{j} \in Y_{j}} \left\{ \left\langle\left(y_{j} - 
				y_{j}^{k}\right), ~ \nabla_{y_{j}} F\left(x^{k + 1} , 
				\phi_{j}^{k}\left(y_{i}^{k}\right) , z^{k}\right)\right\rangle + \frac{\beta_{j}^{k}}
				{2}\left\|y_{j} - y_{j}^{k}\right\|^{2} \right\}, 
			\end{equation*}	
			where $\beta_{j}^{k} \equiv \beta_{j}L_{y_{j}}'\left(x^{k + 1} , v_{j}^{k} , 
			z^{k}\right)$ for some fixed $\beta_{j} > 1$.
		\end{subroutine}}}
\vspace{0.2in}

	To generalize Algorithm~\ref{a:PBIE} to the above recursive splitting, one simply 
	replaces \eqref{Algo:Step1} and \eqref{Algo:Step2} with Subroutines~\ref{a:subroutine x} 
	and \ref{a:subroutine y}, respectively.  
\vspace{0.2in}

    \fcolorbox{black}{Ivory2}{\parbox{17cm}{
		\begin{algorithm}[{\bf Proximal Heterogeneous Block Implicit-Explicit Algorithm}] 
		\label{a:PHeBIE}$~$\\
		    {\bf Initialization.} Choose $\alpha_{i} > 1$  ($i = 1 , 2 , \ldots , P$) , 
		    $\beta_{j} > 1$ ($j = 1 , 2 , \ldots , Q$), $\gamma > 0$ and $\left(x^{0} , y^{0} 
		    , z^{0}\right) \in X \times Y \times Z$. \\
		    {\bf General Step ($k = 0 , 1 , \ldots$)}
			\begin{itemize}
				\item[1.] Update $x^{k + 1}$ according to Subroutine \ref{a:subroutine x}.
				\item[2.] Update $y^{k + 1}$ according to Subroutine \ref{a:subroutine y}.
				\item[3.] Select
					\begin{equation*} 
						z^{k + 1} \in \argmin_{z \in Z} \left\{ F\left(x^{k + 1} , y^{k + 1} 
						, z\right) + \frac{\gamma}{2}\norm{z - z^{k}}^{2} \right\}.
					\end{equation*}
			\end{itemize}
		\end{algorithm}}}
\vspace{0.2in}

	The assumptions for proof of convergence of this algorithm  in the generalized setting 
	take the following form.
\vspace{0.2in}

    \fcolorbox{black}{white}{\parbox{17cm}{
    \begin{assumption} \label{hyp recursive} 
    	Let $\left\{\left(x^{k} , y^{k} , z^{k}\right) \right\}_{k \in \nn}$ be iterates 
    	generated by Algorithm~\ref{a:PHeBIE} with  
		\begin{equation*}
			\left(x^{0} , y^{0} , z^{0}\right) \in X \times Y \times Z = 
			\left(X_{1} \times X_{2} \times \cdots \times X_{P}\right) \times 
			\left(Y_{1} \times Y_{2} \times \cdots \times Y_{Q}\right) \times Z. 
		\end{equation*}
		\begin{enumerate}[(i)]
			\item \label{hyp:1r} $X_{i} \subset \real^{p_{i}}$, $Y_j \subset \real^{q_{j}}$, 
				and $Z \subset \real^{r}$ are nonempty and closed ($0 < p_{i} , q_{j} , r \in 
				\nn$ with $\sum_{i = 1}^{P} p_{i} = p$ and $\sum_{j = 1}^{Q} q_{i} = q$).  
			\item \label{hyp:2r} $F : \real^{p} \times \real^{q} \times \real^{r} \rightarrow 
				\real$ is differentiable on $X \times Y \times Z$ and $\inf F > -\infty$. 
				Moreover, $\nabla_{x_{i}} F$ ($i = 1 , 2 , \ldots , P$) and $\nabla_{y_{j}} 
				F$ ($j = 1 , 2 , \ldots , Q$) are Lipschitz continuous with moduli 
				$L_{x_{i}}\left(u_{i} , y , z\right)$ and $L_{y_{j}}\left(x , v_{j} , 
				z\right)$, respectively. Here $u_{i} \in \real^{p - p_{i}}$ and $v_{j} \in 
				\real^{q - q_{j}}$.  
			\item \label{hyp:3r} The gradient of $F$, $\nabla F$, is Lipschitz continuous on 
				bounded domains in $X \times Y \times Z$. Moreover, there exists 				
				$\lambda_{x_{i}}^{+} , \lambda_{y_{j}}^{+} > 0$ ($i = 1 , 2 , \ldots , P$) 
				($j = 1 , 2 , \ldots , Q$) such that
				\begin{equation*}
					\sup \left\{ L_{x_{i}}\left(u_{i}^{k} , y^{k} , z^{k}\right)~|~k \in \nn 
					\right\} \leq \lambda_{x_{i}}^{+} \quad \mbox{ and } \quad \sup \left\{ 
					L_{y_{j}}\left(x^{k + 1} , v_{j}^{k} ,  z^{k}\right)~|~k \in \nn\right\} 
					\leq \lambda_{y_{j}}^{+}.
				\end{equation*}
			\item \label{hyp:4r} The iterates $\left\{ \left(x^{k} , y^{k} , z^{k}\right) 
				\right\}_{k \in \nn}$ are bounded. 
			\item \label{hyp:5r} The function $\Psi$ defined by \eqref{e:Psi} is a KL 
				function (see Definition \ref{d:KL}).
		\end{enumerate}
		\end{assumption}}}
\vspace{0.2in}

	We now state the generalized convergence result analogous to Theorem~\ref{T:Convergence}.
    \begin{theorem}[Convergence to critical points - recursive] \label{T:Convergence gen}
		Let $\left\{ \left(x^{k} , y^{k} , z^{k}\right) \right\}_{k \in \nn}$ be a sequence 
		generated by Algorithm~\ref{a:PHeBIE} with
		\begin{equation*}
			\left(x^{0} , y^{0} , z^{0}\right) \in X \times Y \times Z = 
			\left(X_{1} \times X_{2} \times \cdots \times X_{P}\right) \times 
			\left(Y_{1} \times Y_{2} \times \cdots \times Y_{Q}\right) \times Z. 
		\end{equation*}
		Suppose that Assumption~\ref{hyp recursive} holds. Then following assertions hold.
        \begin{enumerate}[(a)]
            \item The sequence $\left\{ \left(x^{k} , y^{k} , z^{k}\right) \right\}_{k \in 
            	\nn}$ has finite length, that is,
                \begin{equation*}
                    \sum_{k = 1}^{\infty} \norm{\left(x^{k + 1} , y^{k + 1} , z^{k + 
                    1}\right)- \left(x^{k} , y^{k} , z^{k}\right)} < \infty.
                \end{equation*}
            \item The sequence $\left\{ \left(x^{k}, y^{k}, z^{k}\right) \right\}_{k \in 
            	\nn}$ converges to a point $\left(x^{\ast} , y^{\ast} , z^{\ast}\right)$ 
            	satisfying \eqref{e:Fermat}.
        \end{enumerate}
    \end{theorem}
	\begin{proof}[Proof sketch]
	  The proof of convergence of the multi-block method follows by
	  induction from the proof of the three-block case detailed in 
	  Section~\ref{Sec:Analysis}.  
    \end{proof}
	As mentioned in Section~\ref{Sec:formulation}, the trade-off for the larger step sizes 
	used in recursive blocking is an $(P + Q + 1)$-step sequential algorithm instead of the 
	original $3$-step algorithm. In the next section we explore additional structures that 
	permit parallelization. 

\subsection{Parallelization}
	We show here that the sequential Algorithm~\ref{a:PHeBIE} can be parallelized within the 
	blocks $x$ and $y$ under the following assumption:
\vspace{0.2in}

    \fcolorbox{black}{white}{\parbox{17cm}{
        \begin{assumption} \label{hyp parallel}  
			\begin{enumerate}[(i)]
   				\item\label{hyp parallel 1} For $y \in Y$ and $z \in Z$ fixed, the function 
   					$x \mapsto \nabla_{x} F\left(x , y , z\right)$ is separable in $x$ in the 
   					following sense:  
					\begin{equation} \label{e:nablax F separable}  
				    	\nabla_{x} F\left(x , y , z\right) = \left( g_{1}\left(x_{1} , y , 
				    	z\right) ,  g_{2}\left(x_{2} , y , z\right) , \ldots,  
				    	g_{P}\left(x_{P} , y , z\right)\right),
					\end{equation}
					where $g_{i}\left(\cdot , y , z\right) : X_{i} \to X_{i}$ for $i = 1 , 2 
					, \ldots , P$.
				\item\label{hyp parallel 2} For $x \in X$ and $z \in Z$ fixed, the function 
					$y \mapsto \nabla_{y} F\left(x , y , z\right)$ is separable in $y$ in the 
					following sense:  
					\begin{equation} \label{e:nablay F separable}  
   						\nabla_{y} F\left(x , y , z\right) = \left(h_{1}\left(x , y_{1} , 
   						z\right) ,  h_{2}\left(x , y_{2} , z\right) , \ldots , h_{Q}\left(x , 
   						y_{Q} , z\right)\right), 
					\end{equation}
					where $h_{j}\left(x , \cdot , z\right) : Y_{j} \to Y_{j}$ for $j = 1 , 2 
					, \ldots , Q$.
			\end{enumerate}
	  \end{assumption}}}
\vspace{0.2in}
	 
	An immediate consequence of the above assumption is the following.
	\begin{proposition}[Parallelizable separability] \label{t:parallelization}
		Suppose $F : X \times Y \times Z \rightarrow \real$ satisfies Assumption~\ref{hyp 
		parallel}. Let $\left\{ \left(x^{k} , y^{k} , z^{k}\right) \right\}_{k \in \nn}$ be a 
		sequence generated by Algorithm~\ref{a:PHeBIE}. Then 
		\begin{equation*}
			\nabla_{x_{i}} F\left(\zeta_{i}^{k}\left(x_{i}^{k}\right) , y^{k} , z^{k}\right) 
			= \nabla_{x_{i}} F\left(x^{k} , y^k , z^{k}\right),
		\end{equation*}			
		and
		\begin{equation*}
			\nabla_{y_{j}} F\left(x^{k + 1} , \phi_{j}^{k}\left(y_{j}^{k}\right) , 
			z^{k}\right) = \nabla_{y_{j}} F\left(x^{k + 1} , y^{k} , z^{k}\right).
		\end{equation*}
		Consequently, the modulus of Lipschitz continuity of the gradient of the mapping 
		$x_{i} \mapsto F\left(\zeta_{i}^{k}\left(x_{i}\right) , y^{k} , z^{k}\right)$,  
		$L_{x_{i}}\left(u_{i}^{k} , y^{k} , z^{k}\right)$ is dependent only on $y^{k}$ and 
		$z^k$, thus one can write $L_{x_{i}}\left(y^{k} , z^{k}\right)$ and 
		$L_{x_{i}}'\left(y^{k} , z^{k}\right)$ for the corresponding Lipschitz constants. The 
		same holds for the partial gradients with respect to $y_{j}$, where one can write 
		$L_{y_{j}}\left(x^{k + 1} , z^{k}\right)$ and $L_{y_{j}}'\left(x^{k + 1} , 
		z^{k}\right)$ for the corresponding Lipschitz constants.
	\end{proposition}
	An important  consequence of Proposition~\ref{t:parallelization} is that the successive 
	steps of the respective Subroutines~\ref{a:subroutine x} and \ref{a:subroutine y} can be 
	computed in {\em parallel}. We summarize the results of this section with the following 
	fully decomposable and parallelizable algorithm.
\vspace{0.2in}

    \fcolorbox{black}{Ivory2}{\parbox{17cm}{
		\begin{algorithm}[{\bf Proximal Parallel Heterogeneous Block Implicit-Explicit 
		Algorithm}] 			
		\label{a:PPHeBIE}$~$\\
		    {\bf Initialization.} Choose $\alpha_{i} > 1$ ($i = 1 , 2 , \ldots , P$), 
		    $\beta_{j} > 1$ ($j = 1 , 2 , \ldots , Q$), $\gamma > 0$ and $\left(x^{0} , y^{0} 
		    , z^{0}\right) \in X \times Y \times Z$. \\
		    {\bf General Step ($k = 0 , 1 , \ldots$)}
			\begin{enumerate}[1.]
				\item\label{a:PPHeBIE 1} For each $i = 1 , 2 , \ldots , P$, set 
				  $\alpha_{i}^{k} = \alpha_{i} L_{x_{i}}'\left(y^{k} , z^{k}\right)$ and 
				  select
			    	\begin{equation*}  
					  x_{i}^{k + 1} \in \argmin_{x_{i} \in X_{i}} \left\{ 
					  \left\langle\left(x_{i} - x_{i}^{k}\right),~ \nabla_{x_{i}} 
					  F\left(x^{k} , y^{k} , z^{k}\right)\right\rangle + 
					  \frac{\alpha_{i}^{k}}{2}\left\|x_{i} - x_{i}^{k}\right\|^{2} \right\}.
				    \end{equation*}	
				\item\label{a:PPHeBIE 2} For each $j = 1 , 2 , \ldots , Q$, set 
				    $\beta_{j}^{k} = \beta_{j} L_{y_{j}}'\left(x^{k} , z^{k}\right)$ and 
				    select
					\begin{equation*} 
						y_{j}^{k + 1} \in \argmin_{y_{j} \in Y_{j}} \left\{ 
						\left\langle\left( y_{j} - y_{j}^{k}\right),~ \nabla_{y_{j}} 
						F\left(x^{k + 1} , y^{k} , z^{k}\right)\right\rangle + 
						\frac{\beta_{j}^{k}}{2}\left\|y_{j} - y_{j}^{k}\right\|^{2} \right\}, 
					\end{equation*}	
				\item\label{a:PPHeBIE 3} Select
					\begin{equation*}
						z^{k + 1} \in \argmin_{z \in Z} \left\{ F\left(x^{k + 1} , y^{k + 1} 
						, z\right) + \frac{\gamma}{2}\norm{z - z^{k}}^{2} \right\}.
					\end{equation*}
			\end{enumerate}
		\end{algorithm}}}
\vspace{0.2in}

	We now state the generalized convergence result for the parallel algorithm, analogous to 
	Theorem~\ref{T:Convergence gen}.
    \begin{theorem}[Convergence to critical points - parallel recursive] \label{T:Convergence 
    gen parallel}
		Let $\left\{ \left(x^{k} , y^{k} , z^{k}\right) \right\}_{k \in \nn}$ be a sequence 
		generated by Algorithm~\ref{a:PPHeBIE} with
		\begin{equation*}
			\left(x^{0} , y^{0} , z^{0}\right) \in X \times Y \times Z = 
			\left(X_{1} \times X_{2} \times \cdots \times X_{P}\right) \times 
			\left(Y_{1} \times Y_{2} \times \cdots \times Y_{Q}\right) \times Z. 
		\end{equation*}
		Suppose that Assumptions~\ref{hyp recursive} and \ref{hyp parallel} hold. Then 
		following assertions hold.
        \begin{enumerate}[(a)]
            \item The sequence $\left\{ \left(x^{k} , y^{k} , z^{k}\right) \right\}_{k \in 
            	\nn}$ has finite length, that is,
                \begin{equation*}
                    \sum_{k = 1}^{\infty} \norm{\left(x^{k + 1} , y^{k + 1} , z^{k + 
                    1}\right)- \left(x^{k} , y^{k} , z^{k}\right)} < \infty.
                \end{equation*}
            \item The sequence $\left\{ \left(x^{k} , y^{k} , z^{k}\right) \right\}_{k \in 
            	\nn}$ converges to a point $\left(x^{\ast} , y^{\ast} , z^{\ast}\right)$ 
            	satisfying \eqref{e:Fermat}.
        \end{enumerate}
    \end{theorem}
	\begin{proof}[Proof sketch] 
		The proof of convergence of the parallel multi-block method follows 
		by induction from the proof of the three-block case detailed in
		Section~\ref{Sec:Analysis} and Proposition~\ref{t:parallelization}.  
    \end{proof}

\section{Implementation for Blind Ptychography} \label{Sec:AppPyt}
	We apply the above results to the ptychography problem described in 
	Section~\ref{SSec:Pyt} where the objective function $F$ is given by \eqref{e:F} and the 
	constraint set $C$ by \eqref{e:C}, \eqref{e:support} and \eqref{e:support-amplitude}. The 
	sets $X,Y\subset \comp^n$ decompose into the product of $n$ complex planes. More 
	precisely,  $X \to X_1\times\cdots\times X_n\subset (\comp)^n$ with
 	\begin{equation} \label{e:support i}
		X_i \equiv \begin{cases}
			\left\{x \in \comp~|~|x| \leq R~\right\}, & \mbox{ for } i \in \Ibb_{X}, \\
		    \{0\}, & \mbox{ otherwise},
		\end{cases}
	\end{equation}
	where, again, $\Ibb_{X}$ is the index set corresponding to the support of the probe beam 
	and $R$ is some given amplitude. Similarly, $Y \to Y_1\times\cdots\times Y_n\subset 
	(\comp)^n$ with 
	\begin{equation} \label{e:support-amplitude i}
		Y_i \equiv 
		\begin{cases}
			\left\{ y \in \comp~|~0 \leq \underline{\eta} \leq \left|y\right| \leq 
			\overline{\eta}\right\}, & \mbox{ for } i \in \Ibb_{Y},\\
			\{0\}, & \mbox{ otherwise},
		\end{cases}
	\end{equation}
	where the index set $\Ibb_{Y}$ is the index set for the support of the specimen, and 
	$\underline{\eta}/\overline{\eta}$ are given lower/upper bounds on the intensity of the 
	specimen. We begin by showing that in this setting Assumptions~\ref{hyp recursive} and 
	\ref{hyp parallel} hold. In the context of the more general theory, in Assumption 
	\ref{hyp recursive} for this application we have $P=Q=n$ and $p_i, q_i=2$ for $i=1,2,
	\dots,n$, where $n$ is the number of pixels, and $r=2mn$ where $m$ is the number of 
	images.  
	\begin{proposition}\label{t:Ptychography hypotheses}
	 Let  $F$ be defined by \eqref{e:F} and let the constraint sets $X$, $Y$ and $Z$ be 
	 defined by \eqref{e:C}. Then $F$, together with the constraints $X$, $Y$ and $Z$ 
	 satisfies Assumption~\ref{hyp parallel} and the iterates of Algorithm~\ref{a:PPHeBIE} 
	 satisfy Assumption~\ref{hyp recursive}. Hence Algorithm~\ref{a:PPHeBIE} applied to the 
	 ptychography problem converges to a critical point from any feasible starting point.  
	\end{proposition}
	\begin{proof}	
		There several items from Assumption  \ref{hyp recursive} that are trivial: 
		\begin{itemize}
			\item[]\eqref{hyp:1r} the constraints $X$, $Y$, and $Z$ are clearly nonempty and 
				closed;
			\item[]\eqref{hyp:2r} the objective function $F$ is differentiable;
			\item[]\eqref{hyp:4r}  the generated sequence is bounded since this is a 
			 	feasible point algorithm and the constraint set $C$ is bounded (see Remark 						\ref{r:ptychography assumptions}\eqref{r:compact constraints});
			\item[]\eqref{hyp:3r} the Lipschitz continuity of $\nabla F$ on bounded subsets 
				of $X \times Y \times Z$ follows immediately from the fact that $F$ is 
				$C^{2}$ and the fact that the generated sequence is bounded; 
			\item[]\eqref{hyp:5r} by Proposition~\ref{t:Ptychography semialg} the function 
				$\Psi$ (see \eqref{e:Psi}) is a KL function.
		\end{itemize}
		The only remaining parts needing verification are Lipschitz continuity of the partial 
		gradients in Assumption \ref{hyp recursive}\eqref{hyp:2r} and separability of the 
		gradients in Assumption~\ref{hyp parallel}. The technical details of this calculation 
		are left for an appendix where we show that (in a slight abuse of notation)  		
		\begin{eqnarray}
			\nabla_{x} F\left(x , y , z\right) &=& \left(\nabla_{x_1} F\left(x_1, y , 
			z\right), \nabla_{x_2} F\left(x_2, y, z\right), \dots, \nabla_{x_n} F\left(x_n, 
			y, z\right)\right), \label{e:nablax F ptychography splits}\\
			\nabla_{y} F\left(x , y , z\right) &=& \left(\nabla_{y_1} F\left(x, y_1, 
			z\right), \nabla_{y_2} F\left(x, y_2, z\right), \dots, \nabla_{y_n} F\left(x, 
			y_n, z\right)\right), \label{e:nablay F ptychography splits}
		\end{eqnarray}
		with respective moduli of continuity  
	 	\begin{eqnarray}
			L_{x_i}\left(y , z\right) &=& 2\left(\sum_{j = 1}^{m} 
			S_{j}^{\ast}\left(\overline{y} \odot y\right)\right)_i, \qquad i=1,2,\ldots, 
			n, \label{e:nablaxi F ptychography modulus} \\
	      	L_{y_i}\left(x , z\right) &=& 2\left({\sum_{j = 1}^{m} S_{j}\left(\overline{x} 
	      	\odot x\right)}\right)_i, \qquad i=1,2,\dots, n.
	      	\label{e:nablayi F ptychography modulus}		
	\end{eqnarray}
	Convergence of Algorithm \ref{a:PPHeBIE} applied to critical points of the ptychography 
	problem for any feasible initial guess then follows immediately from Theorem 
	\ref{t:parallelization}.
	\end{proof}
	We note that the partial gradients $\nabla_{x_i} F\left(x , y , z\right)$  (respectively 
	$\nabla_{y_i} F\left(x , y , z\right)$) are with respect to the real and imaginary parts 
	of $x_i\in\comp$ (respectively $y_i\in\comp$), or equivalently with respect to the two-
	dimensional real vectors $x_i\in\real^2$ (respectively $y_i\in\real^2$). So $\nabla_{x_i} 
	F\left(x , y , z\right)$  (respectively $\nabla_{y_i} F\left(x , y , z\right)$) are 
	actually mappings to vectors in $\real^2$ with moduli of Lipschitz continuity 
	$L_{x_i}\left(y , z\right)$ (respectively $L_{y_i}\left(x , z\right)$).
\medskip
	
	The regularization parameters can be determined explicitly by the modulus of Lipschitz 
	continuity of the gradient of $F$ with respect to the isolated blocks of variables $x$ 
	and $y$, respectively. More precisely, for $i=1,2, \dots, n$, we have
	\begin{eqnarray}
		\alpha^{k}_i &=& \alpha L_{x_i}\left(y^{k} , z^{k}\right) = \alpha\left(\sum_{j = 
		1}^{m} S_{j}^{\ast}\left(\overline{y^{k}} \odot y^{k}\right)\right)_i, 
		\label{e:alpha_i ptych}\\
		\beta^{k}_i &=& \beta L_{y_i}\left(x^{k + 1} , z^{k}\right) = \beta \left(\sum_{j = 
		1}^{m} S_{j}\left(\overline{x^{k + 1}} \odot x^{k + 1}\right)\right)_i,
		\label{e:beta_i ptych}
	\end{eqnarray} 
	where $\alpha , \beta > 1$ are arbitrary. 
\medskip

	In drawing the connections to other algorithms in the literature it is helpful to 
	recognize that Steps \eqref{a:PPHeBIE 1} and \eqref{a:PPHeBIE 2} of Algorithm 
	\eqref{a:PPHeBIE} are easily computed projections. Indeed, 
	\begin{align}
		x_i^{k + 1} & \in \argmin_{x_i \in X_i} \left\{ \act{x_i - x_i^{k} , \nabla_{x_i} 
		F\left(x_i^{k} , y^{k} , \bz^{k}\right)} + \frac{\alpha_i^{k}}{2}\norm{x_i - 
		x_i^{k}}^{2} \right\}	\nonumber\\
		& = \argmin_{x_i \in X_i} \left\{ \norm{x_i - \left(x_i^{k} - \frac{2}
		{\alpha_i^{k}}\sum_{j = 1}^{m} \left[\left(S_{j}^{\ast}\left(\overline{y^{k}} \odot 
		y^{k}\right)\right)_i \odot x_i^{k} - \left(S_{j}^{\ast}\left(\overline{y^{k}} \odot 
		z_{j}^{k}\right)\right)_i \right]\right)}^{2} \right\} \nonumber\\
		& = P_{X_i}\left(x_i^{k} - \frac{2}{\alpha_i^{k}}\sum_{j = 1}^{m} 
		\left[\left(S_{j}^{\ast}\left(\overline{y^{k}} \odot y^{k}\right)\right)_i \odot 
		x_i^{k} - \left(S_{j}^{\ast}\left(\overline{y^{k}} \odot z_{j}^{k}\right)\right)_i 
		\right]\right),\label{e:argmin-P_X}
	\end{align}
	where $P_{X_i}$ is the projection onto the constraint set $X_i$. Similarly
	\begin{align}
		y_i^{k + 1} & \in \argmin_{y_i \in Y_i} \left\{ \act{y_i - y_i^{k} , \nabla_{y_i} 
		F\left(x^{k + 1} , y_i^{k} , \bz^{k}\right)} + \frac{\beta_i^{k}}{2}\norm{y_i - 
		y_i^{k}}^{2} \right\} \nonumber\\
    	& = \argmin_{y_i \in Y_i} \left\{ \norm{y_i - \left(y_i^{k} - \frac{2}
    	{\beta_i^{k}}\sum_{j = 1}^{m}\left[\left(S_{j}\left(\overline{x^{k + 1}} \odot x^{k + 
    	1}\right)\right)_i \odot y_i^{k} - \left(S_{j}\left(\overline{x^{k + 1}}\right) \odot 
    	z_{j}^{k}\right)_i \right]\right)}^{2} \right\} \nonumber\\
    	& = P_{Y_i}\left(y_i^{k} - 	\frac{2}{\beta_i^{k}}\sum_{j = 1}^{m} 
    	\left[\left(S_{j}\left(\overline{x^{k + 1}} \odot x^{k + 1}\right)\right)_i \odot 
    	y_i^{k} - \left(S_{j}\left(\overline{x^{k + 1}}\right) \odot z_{j}^{k}\right)_i 
    	\right]\right),	\label{e:argmin-P_Y}	
	\end{align}
	where $P_{Y_i}$ is the projection onto the constraint set $Y_i$. The last step is also a 
	projection step given by
	\begin{align}
		\bz^{k + 1} & \in \argmin_{\bz \in Z} \left\{ F\left(x^{k + 1} , y^{k + 1} , 
		\bz\right) + \frac{\gamma}{2}\norm{\bz - \bz^{k}}^{2} \right\} \nonumber\\
		& = \argmin_{\bz \in Z} \left\{ \sum_{j = 1}^{m} \norm{\left(\frac{2}{2 + \gamma} 
		S_{j}\left(x^{k + 1}\right) \odot y^{k + 1} + \frac{\gamma}{2 + \gamma} 
		z_{j}^{k}\right) - z_{j}}^{2} \right\}\nonumber\\
		& = P_Z\left(\widetilde{\bz}^{k+1}\right), \label{e:argmin-P_Z}
	\end{align}
	where $\widetilde{\bz}^{k+1} \equiv \left(\widetilde{z_1}^{k+1}, \widetilde{z_2}^{k+1}, 
	\dots, \widetilde{z_m}^{k+1} \right)$ for 
	\begin{equation}\label{e:bztilde}
	 \widetilde{z_j}^{k+1}\equiv \frac{2}{2 + \gamma} 
		S_{j}\left(x^{k + 1}\right) \odot y^{k + 1} + \frac{\gamma}{2 + \gamma} 
		z_{j}^{k}, \quad j=1,2,\dots,m.
	\end{equation}
	Since $Z$ is separable, the projection can be written written as 
	\begin{equation}\label{e:PZbzitilde}
	P_Z(\widetilde{\bz}^{k+1})=P_{Z_1}(\widetilde{z_1}^{k+1})\times P_{Z_2}
	(\widetilde{z_2}^{k+1})\times \cdots\times P_{Z_m}(\widetilde{z_m}^{k+1}),
	\end{equation}
	so that $\bz^{k+1} = \left(z_1^{k+1}, z_2^{k+1}, \dots, z_m^{k+1} \right)$ where 
	\begin{equation}\label{e:PZ_iz_itilde}
		z_{j}^{k + 1} \in P_{Z_{j}}\left(\frac{2}{2 + \gamma}S_{j}\left(x^{k + 1}\right) 
		\odot y^{k + 1} + \frac{\gamma}{2 + \gamma}z_{j}^{k}\right), \quad j = 1 , 2 , 
		\ldots , m.
	\end{equation}
	For a given point $z \in \comp^{n}$ the projector onto the set $Z_{j}$ (see 
	\eqref{e:Mj}) is given by \cite{LBL2002}
	\begin{equation} \label{e:PM}
		P_{Z_{j}}\left(z\right) = \iF\left({\hat z}\right) \quad \mbox{ where, for some } 
		\theta \in \left(0 , 2\pi\right], \quad
		{\hat z_{k}} = \begin{cases}
			b_{jk}\frac{[\FFF\left(z\right)]_{k}}{\left|[\FFF\left(z\right)]_{k}\right|},		
			& \left|[\FFF\left(z\right)]_{k}\right| \neq 0, \\
			b_{jk}\mathrm{e}^{i\theta}, & \left|[\FFF\left(z\right)]_{k}\right| = 	0.
		\end{cases}
	\end{equation}
	We summarize this discussion with the following specialization of Algorithm~\ref{a:PBIE} 
	to the blind ptychography problem.
\vspace{0.2in}

    \fcolorbox{black}{Ivory2}{\parbox{17cm}{
		\begin{algorithm}[{\bf Ptychographic PHeBIE}]\label{a:PPHeBIE ptych}$~$\\
		    {\bf Initialization.} Choose $\alpha_{i} > 1$ and $\beta_{i} > 1$ ($i = 1 , 2 , 
		    \ldots , n$), $\gamma > 0$ and $\left(x^{0} , y^{0} , z^{0}\right) \in X \times Y 
		    \times Z$. \\
		    {\bf General Step ($k = 0 , 1 , \ldots$)}
			\begin{enumerate}[1.]
				\item\label{a:PPHeBIE ptych 1} For each $i = 1 , 2 , \ldots , n$, set 
				  $\alpha_{i}^{k} = \alpha_{i}\left(\sum_{j = 1}^{m} 
					S_{j}^{\ast}\left(\overline{y^{k}} \odot y^{k}\right)\right)_i$ and 
					select
			    	\begin{equation*} 
					  x_{i}^{k + 1} \in  P_{X_i}\left(x_i^{k} - \frac{2}{\alpha_i^{k}}\sum_{j 
					  = 1}^{m} \left[S_{j}^{\ast}\left(\overline{y^{k}} \odot y^{k}\right)_i 
					  \odot x_i^{k} - S_{j}^{\ast}\left(\overline{y^{k}} \odot 
					  z_{j}^{k}\right)_i \right]\right).
				    \end{equation*}	
				\item\label{a:PPHeBIE ptych 2} For each $i = 1 , 2 , \ldots , n$, set 
				    $\beta_{i}^{k} = \beta_{i} \left(\sum_{j = 1}^{m} 
				    S_{j}\left(\overline{x^{k + 1}} \odot x^{k + 1}\right)\right)_i$
				    and select
					\begin{equation*} 
						y_{i}^{k + 1} \in P_{Y_i}\left(y_i^{k} - 	\frac{2}
						{\beta_i^{k}}\sum_{j = 1}^{m} 	      
						\left[\left(S_{j}\left(\overline{x^{k + 1}} \odot x^{k + 
						1}\right)\right)_i \odot y_i^{k} - 					
						\left(S_{j}\left(\overline{x^{k + 1}}\right) \odot z_{j}^{k}\right)_i 
						\right]\right).
					\end{equation*}	
				\item\label{a:PPHeBIE ptych 3} For each $j=1,2,\dots,m$ select
					\begin{equation*} 
						z_j^{k + 1} \in P_{Z_{j}}\left(\frac{2}{2 + \gamma}S_{j}\left(x^{k + 
						1}\right) \odot y^{k + 1} + \frac{\gamma}{2 + 
						\gamma}z_{j}^{k}\right).
					\end{equation*}
			\end{enumerate}
		\end{algorithm}}}
\vspace{0.2in}

	Convergence of Algorithm~\ref{a:PPHeBIE ptych} to critical points has already been 	
	established in Proposition \ref{t:Ptychography hypotheses}.
	
\subsection{Relation to Current State-of-the-Art Algorithms}
	It is helpful to see Algorithm~\ref{a:PPHeBIE ptych} in the context of two other blind 
	ptychographic reconstruction algorithms, popular in the literature, namely, the methods 
	of Thibault \etal \cite{TDBMP2009}, and {Maiden and Rodenburg \cite{MR2009}. We show 
	that these algorithms should not be expected to converge in general to a fixed point.  
	However, the connection to Algorithm~\ref{a:PPHeBIE ptych} and the attendant analysis 
	immediately suggests how the methods of Thibault and Maiden and Rodenburg can be adjusted 
	for greater stability or speed (or both). On the other hand, understanding these two 
	methods in the context of the more general Algorithm~\ref{a:PPHeBIE} points the way to 
	different constructions and compositions of the three basic steps of either Algorithm 
	\ref{a:PBIE} or \ref{a:PPHeBIE} for more efficient procedures.  The analysis of such 
	variants would then follow along the lines of the analytical methodology presented here.  

\subsubsection{Thibault \etal \cite{TDBMP2009}}\label{s:Thibault}
	In order to explain the scheme suggested in \cite{TDBMP2009} we first recall the 
	definition of 
	\begin{equation} \label{eq:Thibauldset1}
		Z \equiv Z_{1} \times Z_{2} \times \cdots \times Z_{m} \subset \comp^{n \times m}.
	\end{equation}	 
	Define the set $D$ on the product space $\comp^{n \times m}$:
	\begin{equation} \label{eq:Thibauldset2}
		D \equiv D_{1} \times D_{2} \times \cdots \times D_{m} \subset \comp^{n \times m},	
	\end{equation}
	where
	\begin{equation}
		D_{j} \equiv \left\{ z_j~|~z_j=S_{j}\left(x\right) \odot y\mbox{ for some } x , y \in 
		\comp^{n} \right\}.
	\end{equation}
	If it were possible to compute the projection onto the set $D$ (no closed form exists), 
	then the Douglas--Rachford algorithm \cite{DougRach56,LionsMercier,BCL2004} could be 
	applied to solve the feasibility problem:  
	\begin{equation*}
  		\text{Find }\overline{x} \in D \cap Z.
	\end{equation*}
	More precisely, we have the following algorithm.
\vspace{0.2in}

    \fcolorbox{black}{Ivory2}{\parbox{17cm}{
	\begin{algorithm}[{\bf Douglas--Rachford for ptychography}] \label{a:DR ptych}$~$\\
    {\bf Initialization.} $\left(x^{0} , y^{0} , \bz^{0}\right) \in X \times Y \times Z$. \\
    {\bf General Step ($k = 0 , 1 , \ldots$)}
	\begin{enumerate}[1.]
		\item\label{a:DR ptych 1} Select an approximation $\bf{v}^k$ to some element from $ 	
			P_D \bz^k$.
		\item\label{a:DR ptych 2} Select
			\begin{equation*} 
				\hat\bz^{k + 1} \in P_{Z}\left(2 \bf{v}^k-\bz^k \right).
			\end{equation*}
		\item\label{a:DR ptych 3} Set
			\begin{equation} \label{e:ThibaultDM}
				\bz^{k + 1} = \bz^{k}+ \hat\bz^{k +1} -\bf{v}^k.
			\end{equation}
	\end{enumerate}
	\end{algorithm}}}
\vspace{0.2in}
%
%\vspace{0.2in}
%    \fcolorbox{black}{Ivory2}{\parbox{17cm}{
%	\begin{algorithm}[{\bf Alternating projections for ptychography}] \label{a:AP ptych}$~$\\
%%     {\bf Input.} $\gamma > 0$. \\ 
%    {\bf Initialization.} $\left(x^{0} , y^{0} , \bz^{0}\right) \in X \times Y \times Z$. \\
%    {\bf General Step ($k = 0 , 1 , \ldots$)}
%	\begin{enumerate}[1.]
%		\item\label{a:AP ptych 1} Select an approximation $\bf{v}^k$ to an element from $ P_D \bz^k$.
%		
%		\item\label{a:AP ptych 2} Select
%			\begin{equation*} 
%				\bz^{k + 1} \in P_{Z}\left(\bf{v}^k\right).
%			\end{equation*}
%%		\item\label{a:AP ptych 3} Set
%%			\begin{equation} \label{e:ThibaultAP}
%%				\bz^{k + 1} = \hat\bz_{j}^{k +1}.
%%			\end{equation}
%	\end{enumerate}
%	\end{algorithm}}}
%\vspace{0.2in}

	As noted above, no closed form exists for the projection onto the set $D$.  The method of 
	\cite{TDBMP2009} is an approximate Douglas--Rachford algorithm for set-feasibility with 
	the following subroutine serving as an approximation to some element from the projector 
	$P_D$. We describe the subroutine below as an approximation to the projector, however 
	there has been no analysis to estimate exactly how good, or in what sense, it is an 
	approximation, hence the qualifier {\em heuristic}.  
\vspace{0.2in}

  \fcolorbox{black}{Ivory2}{\parbox{17cm}{
		\begin{subroutine}[{\bf Heuristic approximation to $P_D$}] %[{\bf Thibault [approximate] Projection $P_D$}] 
		\label{a:subroutinePD}$~$\\
		    {\bf Input.} $x^{k}\in\comp^n,y^{k}\in\comp^n,\bz^{k}\in\comp^{n\times m},
		    	\Lambda\in\{1,2,3,\dots\}$\\
		    {\bf Initialization.} Define $\hat x^{0}\equiv x^{k}$, $\hat y^{0}\equiv y^{k}$.
		    \\
		    {\bf General Step ($l = 0 , 1 , \ldots,\Lambda $).}
			\begin{enumerate}[1.]
			\item\label{a:subroutinePD 1} Define $\alpha^l\in\real^m$ by
				\[\alpha_{i}^l\equiv L_{x_i}(\hat y^{l},z^k)=
				2\left(\sum_{j = 1}^{m}S_{j}^{*}\left(\overline{\hat y^{l}}\odot \hat 
				y^{l}\right)\right)_i,\quad\text{for }i=1,2,\dots,m, \]
				and update $\hat x^{l + 1}$ by
				\begin{align*}
					\hat x_i^{l+1} & =\frac{2}{ \alpha_{i}^l}\left(\sum_{j = 1}^{m} 
					S_{j}^{\ast}\left(\overline{\hat y^{l}} \odot z_{j}^{k}\right)\right)_i.
				\end{align*}
			\item\label{a:subroutinePD 2} Define $\beta^l\in\real^m$ by
				\[
				\beta_{i}^l\equiv L_{y_i}(\hat x^{l},z^k)=
				2\left({\sum_{j = 1}^{m} S_{j}\left(\overline{\hat x^{l}} \odot \hat 
				x^{l}\right)}\right)_i,\quad\text{for }i=1,2,\dots,m, \]
				and update $\hat y^{k + 1}$ by
				\begin{align*}
					\hat y_i^{l+1} & = \frac{2}{ \beta_{i}^l}\left({\sum_{j = 1}^{m} 
					S_{j}\left(\overline{\hat x^{l}}\right)\odot z_{j}^{k}}\right)_i.
				\end{align*}
			\end{enumerate}
		    {\bf Final Step.}	Define $x^{k+1}\equiv \hat x^{\Lambda+1}$, $y^{k+1}\equiv \hat 
		    y^{\Lambda+1}$ and set
			\begin{equation}\label{e:ThibauldPD}
				\widetilde{\bf{v}}^{k+1} \equiv \left(S_{1}\left(x^{k + 1}\right) \odot y^{k 
				+ 1},\quad\cdots\quad, S_{m}\left(x^{k + 1}\right) \odot y^{k + 1}\right).
		\end{equation}
		\end{subroutine}}}
\vspace{0.2in}

	The method of \cite{TDBMP2009} is Algorithm~\ref{a:DR ptych} with Step \ref{a:DR ptych 1} 
	replaced with the computation of $\widetilde{\bf{v}}^k$ via Subroutine 
	\ref{a:subroutinePD}. Subroutine \ref{a:subroutinePD}, in turn,  can be cast within our 
	framework. Step \ref{a:PPHeBIE ptych 1} (respectively Step \ref{a:PPHeBIE ptych 2}) of 
	Algorithm \eqref{a:PPHeBIE ptych} with $X_i=\comp$ (respectively $Y_i=\comp$) for each 
	$i=1,2,\dots, n$ is equivalent to Step \ref{a:subroutinePD 1} (respectively Step 
	\ref{a:subroutinePD 2}) of Subroutine \ref{a:subroutinePD}.  
%	Then, if instead of applying the Douglas--Rachford algorithm with the heuristic 
%	Subroutine \ref{a:subroutinePD} for the projector $P_D$, we applied the alternating 
%	projections Algorithm \ref{a:AP ptych} with Subroutine \ref{a:subroutinePD} and 
%	$\Lambda=0$, we would exactly recover Algorithm \ref{a:PPHeBIE ptych} with parameters 
%	$\alpha_i=\beta_i=1~ (i=1,2,\dots,n)$ and $\gamma=0$.  The convergence analysis for 
%	Algorithm \ref{a:PPHeBIE ptych} does not permit these parameter values, but any positive 
%	perturbation is permitted. More appropriate would be an analysis of Algorithm 
%	\ref{a:PPHeBIE ptych} with {\em errors}, but this is beyond the scope of this work. 
	\begin{remark} \label{r:Thibault}
		Some further remarks on the method of Thibault \etal are in order.
		\begin{enumerate}[(i)]
		    \item\label{r:Thibault shadows} In an implementation of Thibault \etal, one 
		    	monitors $x^{k}$ and $y^{k}$ (\ie the object and illumination function) 
		    	rather than the iterate $\bz^{k}$ itself. Since $x^{k}$ and $y^{k}$ are 
		    	obtained during the computation of the so called \emph{shadow iterates}, 
		    	$P_D\bz^k$, this can be interpreted as implicit monitoring of the 
		    	\emph{shadow sequence}, $\left(P_D\bz^k\right)_{k=1}^\infty$.
				% This is similar to monitoring the \emph{shadow sequence}, as $x^{k}$ and $y^{k}$ are computated during the computation of $P_{D}\bz^{k }$, which is a kind of regularization. 
	        \item\label{r:Thibault instability} The Douglas--Rachford methods is known to be 
	        	sensitive to small perturbations in the constraint sets. In particular, if 
	        	the intersection $D\cap Z$ is empty (not at all an improbable event with 
	        	noisy, miss-specified data), then the Douglas--Rachford {\em cannot} converge 
	        	\cite{BCL2004}. The relaxation of the Douglas--Rachford algorithm studied in 
	        	\cite{RAAR, L2008} is one possibility for addressing this.
		\end{enumerate}
	\end{remark}

\subsubsection{Maiden and Rodenburg \cite{MR2009}}\label{s:Rodenburg}	
    In comparison to the other algorithms presented, the distinctive feature of the
    method of Maiden and Rodenburg \cite{MR2009} is that only a single magnitude measurement
    in used in each step. Their method can be described as follows.
\vspace{0.2in}

    \fcolorbox{black}{Ivory2}{\parbox{17cm}{
		\begin{algorithm}[{\bf Maiden and Rodenburg}]\label{a:PPHeBIE-Rodenburg}$~$\\
		    {\bf Initialization.} Choose $\alpha_{i}=\alpha \geq 2$ and $\beta_{i}=\beta \geq 
			    2$ ($i = 1 , 2 , \ldots , n$). Fix the mapping $\mathbb{I}:\mathbb{N}\mapsto 
			    \{1,2,\dots,m\}$ where the cardinality of the preimage of any $j\in \{1,2,
			    \dots,m\}$ is infinite. Choose $\left(x^{0}, y^{0}, z^{0}\right) \in X \times 
			    Y \times Z_{\mathbb{I}(0)}$. \\
		    {\bf General Step ($k = 0 , 1 , \ldots$)}
			\begin{enumerate}[1.]
				\item\label{a:PPHeBIE-Rodenburg 1} Set $\alpha^{k} = \alpha\left\|\sum_{j = 	
					1}^{m} S_{j}^{\ast}\left(\overline{y^{k}} \odot y^{k}\right)\right\|
					_\infty$ and, for each $i = 1 , 2 , \ldots , n$, select
			    	\begin{equation*} 
						x_{i}^{k + 1} \in  P_{X_i}\left(x_i^{k} - \frac{2}{\alpha^{k}}\sum_{j 
						= 1}^{m} \left[S_{j}^{\ast}\left(\overline{y^{k}} \odot 
						y^{k}\right)_i \odot x_i^{k} - S_{j}^{\ast}\left(\overline{y^{k}} 
						\odot z^{k}\right)_i \right]\right).
			    \end{equation*}	
			\item\label{a:PPHeBIE-Rodenburg 2} Set $\beta^{k} = \beta\left\|\sum_{j = 1}^{m} 
				    S_{j}\left(\overline{x^{k}} \odot x^{k}\right)\right\|_\infty$ and, for 
				    each $i = 1 , 2 , \ldots , n$,  select
					\begin{equation*} 
						y_{i}^{k + 1} \in P_{Y_i}\left(y_i^{k} - \frac{2}{\beta^{k}}\sum_{j 
						= 1}^{m} \left[\left(S_{j}\left(\overline{x^{k}} \odot 
						x^{k}\right)\right)_i \odot y_i^{k} - 
						\left(S_{j}\left(\overline{x^{k}}\right) \odot z^{k}\right)_i 
						\right]\right).
					\end{equation*}	
			\item\label{a:PPHeBIE-Rodenburg 3} Select
					\begin{equation*} 
						z^{k + 1} \in P_{Z_{\mathbb{I}(k+1)}}\left(S_{j}\left(x^{k + 
						1}\right) \odot y^{k + 1}\right).
					\end{equation*}
			\end{enumerate}
		\end{algorithm}}}
\vspace{0.2in}

%	Compared to \eqref{AlgoPty:Step1} \eqref{AlgoPty:Step2} the Rodenburg update represents
%	a \emph{simultaneous} update on $x$ and $y$ in the case $m=1$ and $X=Y=\comp^{n}$.

	\begin{remark}\label{r:Rodenburg}
		In the context of Algorithm \ref{a:PPHeBIE ptych} several features of Algorithm 
		\ref{a:PPHeBIE-Rodenburg} are worth noting. 
		\begin{enumerate}[(i)]
			\item As established in Sections \ref{Sec:Convergence}  and \ref{Sec:AppPyt}, the 
				scalings $\alpha^k$ and $\beta^k$ in  Steps \ref{a:PPHeBIE-Rodenburg 1} and 
				\ref{a:PPHeBIE-Rodenburg 2} of Algorithm~\ref{a:PPHeBIE-Rodenburg} are 
				Lipschitz constants of the partial gradient of $F$ defined by \eqref{e:F}  on 
				the {\em entire} $x$ and $y$ blocks. This could be refined by using the 
				scalings $\alpha_i^k$ given in Steps \ref{a:PPHeBIE ptych 1} and 
				\ref{a:PPHeBIE ptych 2} of Algorithm \ref{a:PPHeBIE ptych}. 
			\item Steps \ref{a:PPHeBIE-Rodenburg 1} and \ref{a:PPHeBIE-Rodenburg 2} of 
				Algorithm~\ref{a:PPHeBIE-Rodenburg}can be performed in {\em parallel} since 
				the $y$ update does not use information from the $x$ update as in Algorithm 	
				\ref{a:PPHeBIE ptych}. 
			\item As Algorithm \ref{a:PPHeBIE-Rodenburg} is essentially a cyclic projection 
				algorithm, in practice one should expect the iterates to cycle. 
		\end{enumerate}
	\end{remark}

\section{Numerical Examples} \label{Sec:Implementing}
	To illustrate the differences between the various algorithms developed above, in Section 
	\ref{s:synthetic} we compare algorithm performance on synthetic data where the problem 
	``difficulty'' is relatively well controlled (and the answer, shown in  
	Figure~\ref{f:Gaens}, is known) and in Section \ref{s:experimental} we compare algorithm 
	performance on experimental data reported in \cite{Wilke}. In both the synthetic and 
	experimental demonstrations we compare different algorithms: 
	\begin{enumerate}
  		\item PHeBIE-I: Algorithm~\ref{a:PBIE} specialized to ptychography with $\gamma = 
  			1\mbox{e-}30$.
	  	\item PHeBIE-II: Algorithm~\ref{a:PPHeBIE ptych} with $\gamma = 1\mbox{e-}30$.
  		% \item Thibault-AP: Alternating Projections algorithm~\ref{a:AP ptych} with Step \ref{a:AP ptych 1} computed via Suboroutine~\ref{a:subroutinePD} with $\Lambda = 3$.
		\item Thibault \cite{TDBMP2009}: Algorithm~\ref{a:DR ptych} with Step \ref{a:DR ptych 
			1} computed via Suboroutine~\ref{a:subroutinePD} with $\Lambda = 3$.
		\item Maiden and Rodenburg~\cite{MR2009}: Algorithm~\ref{a:PPHeBIE-Rodenburg}.
	\end{enumerate}

\subsection{Synthetic data}\label{s:synthetic}
	\begin{figure*}[h]
		\begin{center}
			\includegraphics[height=4.5cm]{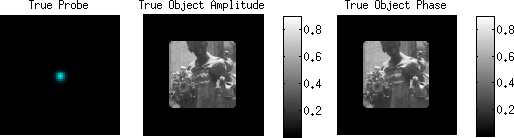}
		\end{center}
		\caption{The true probe and object used in the generation of simulated dataset.} 	
		\label{f:Gaens}
	\end{figure*}

% The object and probe constraints, as well as the over-restrictive probe constraint, are shown in Figure~\ref{f:constraints}.
% 
% \begin{figure}
%   \begin{center}
%    \includegraphics[height=5cm]{20140429Results/objectConstraint.png}
%    \hspace{5ex}
%    \includegraphics[height=5cm]{20140429Results/probeConstraint.png}
%    \hspace{5ex}
%    \includegraphics[height=5cm]{20140429Results/restrictedProbeConstraint.png}
%  \end{center}
%   \caption{The object constraint (left), probe constraint (centre), and the over-restricted 
% probe constraint (right). The support of the object/probe is drawn in white.}\label{f:constraints}
% \end{figure}
% 
	Let $x$ (respectively $y$) denote the true probe (true object). For the noiseless 
	simulated data, we compute the measured data vectors using 
	$$b_j=|\FFF(S_j(x)\odot y)|\text{ for }j=1,2,\dots,m.$$
	For the simulated data with noise, we use Poisson noise with mean/variance $\lambda=2$.
\medskip

In typical ptychography experiments 
one more or less knows a priori what the probe looks like, though its precise structure, due to instrumentation 
aberrations, is unknown.
The object, on the other hand, is assumed to be completely unknown except for certain qualitative properties, for example,
that it is not absorbing.
For the simulated data, the initial probe estimate consists
of a circle of radius slightly larger than the true probe having constant amplitude and phase. 
Objects are initialized with a random initial guess. 
We demonstrate the stability of the algorithms in the results shown in Table \ref{t:gaensPupil} by purposely 
constraining  the pupil to be smaller than the true pupil.  This is not an unreasonable scenario since in practice
the true pupil is not known.  

Consistent with existing literature, we run several iterations of each algorithm 
without updating the probe to obtain a better initial object guess. The results of this ``warm-up" 
procedure are then used as initial point 
$(x^0,y^0,\mathbf z^0)$ for the main algorithm of which 300 iterations were performed. 
Experimentally, the ``warm-up" procedure could also be accomplished with an ``empty'' beam data set consisting of 
beam images taken without specimen.  
\medskip

Where convenient, we use $u^k$ to denote $(x^k,y^k,\mathbf z^k)$. Random trials of each problem instance 
were performed with random object initializations. Tables~\ref{t:gaensNoNoise}, \ref{t:gaensPoisson} and \ref{t:gaensPupil} report the average, 
and in brackets, the worst result for the following statistics. 
 \begin{enumerate}[1.]
  \item The final value of the least-squares objective given by \eqref{e:F}.
  \item The square of the norm of the change between the final two iterations, \ie $\|u^{300}-u^{299}\|^2$. 
  \item The Root-mean-squared error of the final object and probe as described in \cite{guizar2008efficient}. 
The error is computed up to translation, a global phase shift and a global scaling factor.\footnote{Computed 
using code written Mauel Guizar available online at 
\url{http://www.mathworks.com/matlabcentral/fileexchange/18401-efficient-subpixel-image-registration-by-cross-correlation}}
  \item The \emph{$R$-factor} at iteration $300$, where 
    \begin{equation}\label{e:R-factor}
R\text{-factor}^k=\frac{\sum_{j=1}^m\|b_j-S_j(x^{k})\odot y^{k}\|}{\sum_{j=1}^mb_j}.       
    \end{equation}
  As in \eqref{e:Mj}, $b_j$ denotes the experimental observations.
  \item The total time (seconds) for the ``warm-up" and main algorithm.
 \end{enumerate}

\begin{remark}[Error metrics]\label{r:errorMetrics}
Theorem~\ref{T:Convergence gen}(a) guarantees that the difference between the iterates of 
Algorithm \ref{a:PBIE} and Algorithm \ref{a:PPHeBIE} converge in norm to zero. 
% Due to degeneracies in the model formulation given in Section~\ref{SSec:Pyt}, it necessary to consider 
% RMS errors invariant under the transformation described above. Further comments regarding this can be 
% found in \cite{MR2009,guizar2008efficient}. 
To compute the RMS-error a knowledge of the true object 
and probe are required, which in real applications are not known. The $R$-factor can still be evaluated 
in experimental settings (see Figure \ref{f:comp convergence exp}) and used as a measure of quality of 
the reconstruction, though the theoretical behavior of this metric is not covered by our analysis.  
\end{remark}

\begin{table}[htbp]
 \caption{Average (worst) results for noiseless simulated data.}\label{t:gaensNoNoise}
 \vspace{5pt}
 \resizebox{\textwidth}{!}{
  \begin{tabular}{l
                  D{.}{.}{4.2}@{\extracolsep{0.5ex}}D{.}{.}{5.2}
                  @{\extracolsep{5ex}}D{.}{.}{2.4}@{\extracolsep{0.5ex}}D{.}{.}{3.4}
                  *{3}{@{\extracolsep{5ex}}D{.}{.}{1.4}@{\extracolsep{1ex}}D{.}{.}{2.4}}
                  @{\extracolsep{5ex}}D{.}{.}{3.2}@{\extracolsep{0ex}}D{.}{.}{5.2}@{\hspace{1.5ex}}
                  } \hline
   Algorithm                   & \multicolumn{2}{c}{$F(u^{300})$}   & \multicolumn{2}{c}{$\|u^{300}-u^{299}\|^2$} & \multicolumn{2}{c}{RMS-Object}       & \multicolumn{2}{c}{RMS-Probe}        & \multicolumn{2}{c}{R-factor$^{300}$}         & \multicolumn{2}{c}{Time (s)}          \\ \hline
PHeBIE-I 	&  99.63 & (126.64) 	&  0.5931 & (0.9383) 	&  0.0410 & (0.0461) 	&  0.0155 & (0.0222) 	&  0.0131 & (0.0154) 	&  913.75 & (925.85) \\
PHeBIE-II 	&  70.76 & (77.17) 	&  0.2210 & (0.3522) 	&  0.0423 & (0.0471) 	&  0.0081 & (0.0154) 	&  0.0101 & (0.0108) 	&  636.74 & (652.17) \\
Rodenburg \& Madien 	&  948.13 & (1499.11) 	&  5.5164 & (8.4885) 	&  0.0542 & (0.0590) 	&  0.0952 & (0.1714) 	&  0.0350 & (0.0419) 	&  1178.21 & (1198.72) \\
%Thibault-AP 	&  91.28 & (142.79) 	&  0.4497 & (1.4799) 	&  0.0426 & (0.0482) 	&  0.0086 & (0.0163) 	&  0.0114 & (0.0144) 	&  904.10 & (920.22) \\
Thibault 	&  4347.08 & (4554.28) 	&  28.8622 & (34.4422) 	&  0.0515 & (0.0642) 	&  0.0240 & (0.0378) 	&  0.0244 & (0.0264) 	&  875.94 & (887.76) \\
  \hline
  \end{tabular}
 }\\
\end{table}

\begin{table}[htbp]
 \caption{Average (worst) results for simulated data with Poisson noise.}\label{t:gaensPoisson}
 \vspace{5pt}
 \resizebox{\textwidth}{!}{
  \begin{tabular}{l
                  D{e}{e}{5.3}@{\extracolsep{2ex}}D{e}{e}{6.3}
                  @{\extracolsep{5ex}}D{.}{.}{5.4}@{\extracolsep{0.5ex}}D{.}{.}{6.4}
                  *{3}{@{\extracolsep{5ex}}D{.}{.}{1.4}@{\extracolsep{0ex}}D{.}{.}{2.4}}
                  @{\extracolsep{5ex}}D{.}{.}{3.2}@{\extracolsep{0ex}}D{.}{.}{5.2}@{\hspace{1.5ex}}
                  } \hline
   Algorithm                   & \multicolumn{2}{c}{$F(u^{300})$}   & \multicolumn{2}{c}{$\|u^{300}-u^{299}\|^2$} & \multicolumn{2}{c}{RMS-Object}       & \multicolumn{2}{c}{RMS-Probe}        & \multicolumn{2}{c}{R-factor$^{300}$}         & \multicolumn{2}{c}{Time (s)}          \\ \hline
PHeBIE-I 	&  1.4415e+07 & (6.9222e+07) 	&  4.1504 & (14.0823) 	&  0.1928 & (0.6840) 	&  0.1896 & (0.7084) 	&  0.3499 & (1.2698) 	&  899.30 & (933.54) \\
PHeBIE-II 	&  1.4364e+07 & (6.8972e+07) 	&  521.9450 & (2600.9689) 	&  0.2807 & (0.9940) 	&  0.2537 & (0.9746) 	&  0.4001 & (1.5135) 	&  685.67 & (714.21) \\
Rodenburg \& Madien	&  6.7894e+04 & (3.1414e+05) 	&  14633.8000 &(61868.3743) &  0.2654 & (0.9996) 	&  0.3205 & (0.9507) 	&  0.3827 & (1.2814) 	&  1168.36 & (1177.71) \\
%Thibault-AP	&  1.4384e+07 & (6.9070e+07) 	&  30.4566 & (143.8009) 	&  0.2710 & (0.9887) 	&  0.2332 & (0.8940) 	&  0.4039 & (1.5324) 	&  915.60 & (1058.90) \\
Thibault 	&  1.4520e+07 & (6.9688e+07) 	&  247.3130 & (976.2039) 	&  0.2476 & (1.0000) 	&  0.0700 & (0.2498) 	&  0.1748 & (0.5686) 	&  868.07 & (892.19) \\
  \hline
  \end{tabular}
 }\\
\end{table}

\begin{table}[htbp]
 \caption{Average (worst) results for noiseless simulated data with over-restrictive pupil constraint.}\label{t:gaensPupil}
 \vspace{5pt}
 \resizebox{\textwidth}{!}{
  \begin{tabular}{l
                  D{.}{.}{6.2}@{\extracolsep{0.5ex}}D{.}{.}{7.2}
                  @{\extracolsep{5ex}}D{.}{.}{3.4}@{\extracolsep{0.5ex}}D{.}{.}{4.4}
                  *{3}{@{\extracolsep{5ex}}D{.}{.}{1.4}@{\extracolsep{0ex}}D{.}{.}{2.4}}
                  @{\extracolsep{5ex}}D{.}{.}{3.2}@{\extracolsep{0ex}}D{.}{.}{5.2}@{\hspace{1.5ex}}
                  } \hline
   Algorithm                   & \multicolumn{2}{c}{$F(u^{300})$}   & \multicolumn{2}{c}{$\|u^{300}-u^{299}\|^2$} & \multicolumn{2}{c}{RMS-Object}       & \multicolumn{2}{c}{RMS-Probe}        & \multicolumn{2}{c}{R-factor$^{300}$}         & \multicolumn{2}{c}{Time (s)}          \\ \hline
PHeBIE-I 	&  25653.20 & (25656.12) 	&  0.1474 & (0.1682) 	&  0.0443 & (0.0501) 	&  0.0492 & (0.0494) 	&  0.2936 & (0.2937) 	&  959.26 & (1108.53) \\
PHeBIE-II 	&  25653.80 & (25660.13) 	&  0.0622 & (0.0852) 	&  0.0314 & (0.0356) 	&  0.0496 & (0.0499) 	&  0.2936 & (0.2937) 	&  632.61 & (645.10) \\
Rodenburg \& Madien	&  3987.67 & (4413.77) 	&  6.2921 & (16.3055) 	&  0.0689 & (0.0760) 	&  0.0550 & (0.0570) 	&  0.2834 & (0.2839) 	&  1190.19 & (1309.23) \\
%Thibault-AP 	&  25644.30 & (25649.88) 	&  5.6868 & (5.6944) 	&  0.9355 & (0.9356) 	&  0.0493 & (0.0494) 	&  0.2936 & (0.2936) 	&  865.08 & (986.84) \\
Thibault 	&  306602.00 & (378554.38) 	&  219.2110 & (267.9303) 	&  0.9432 & (0.9437) 	&  0.1898 & (0.2435) 	&  0.3634 & (0.39160) 	&  897.67 & (962.45) \\
  \hline
  \end{tabular}
 }\\
\end{table}

For the noiseless simulated dataset, the quality of the reconstructed object and probe from
each of the methods examined are comparable. This applies to the quantitative error
metrics recorded in Table~\ref{t:gaensNoNoise}, as well as to a visual comparison of  
reconstructions (not shown). In the absence of noise, all the methods
examined worked well.  
% As our analysis suggests, Algorithm \ref{a:AP ptych} with Step \ref{a:AP ptych 1} replaced by Subroutine \ref{a:subroutinePD}, that is, the method of Thibault with alternating projections instead of the Douglas--Rachford algorithm (what we call Thibault-AP), performs nearly the same as Algorithm \ref{a:PPHeBIE ptych}.  
It is worth noting, that Algorithm \ref{a:PPHeBIE ptych} was significantly faster than all the other methods. 
\medskip

With the addition of Poisson noise, the quality of the reconstructed
objects and probes deteriorates. The error metrics are mixed, and no clear ``winner'' 
emerges from the values reported in Table~\ref{t:gaensPoisson}.
The method of Thibault could be expected to 
be more unstable since it is based on the Douglas--Rachford algorithm, it has the advantage of pushing 
past local minima that might otherwise trap Algorithm \ref{a:PPHeBIE}.  
\medskip

The results for a improperly specified pupil constraint (too small) demonstrate the relative stability of the respective 
methods.  The method of Thibault \etal was the most sensitive to the over-restrictive pupil 
constraint and performed the worst. This is expected since the modeling errors lead to {\em inconsistency} of the 
underlying feasibility problem:  it is well known that Douglas--Rachford does not have fixed points for 
inconsistent feasibility problems \cite{BCL2004}. Visually, the method of Thibault was not able to recover any 
semblance of the true solution.   Algorithms \ref{a:PHeBIE} and \ref{a:PPHeBIE} are clearly more robust.  Visual comparisons also bear this out.
\medskip

Tables~\ref{t:gaensNoNoise}, \ref{t:gaensPoisson} and \ref{t:gaensPupil} suggest that is it not appropriate to 
compare the final objective value and stepsize of the various algorithms directly. Significant variability 
was exhibited in these metrics between the methods, despite all recording having similar 
RMS and $R$-factor errors in the ideal case of noiseless data. 

\subsection{Experimental Data}\label{s:experimental}
In this section we examine the four algorithms applied to an experimental data set (from \cite{Wilke}) 
in which the actual illumination function and specimen are unknown. 
The reconstructed illumination functions and specimens obtained from the four algorithms are shown in 
Figure~\ref{f:comp recon exp}. By visual inspection, the reconstructions are of comparable quality, 
with the exception of the results from method of Madien and Rodenburg, which is of noticeably poorer quality.
\medskip

In Figure~\ref{f:comp convergence exp} compare two error metrics as a function of number of iterations for the four 
algorithms. The first graph, Figure~\ref{f:comp convergence exp}(a)  shows the norm of the difference of 
successive iterates, which, for PHeBIE, is 
the only quantity guaranteed to converge to zero by the theory we have developed above.  The second graph, 
Figure~\ref{f:comp convergence exp}(b) shows the $R$-factor which, as discussed in Remark~\ref{r:errorMetrics} is 
computable in experimental settings. The best performance, with respect to both of these metrics, were observed 
for the fully decomposed parallel PHeBIE-II (Algorithm \ref{a:PPHeBIE ptych}). 
We do not make any direct comparison with the reconstructions in \cite{Wilke}, 
however, because there the authors implement routines beyond the scope of our theory. 
A more complete benchmarking study on experimental data is forthcoming.
\begin{figure}[h]
  \centering
  \begin{subfigure}[b]{\linewidth} 
     \centering
     \includegraphics[width=15cm, height=7cm]{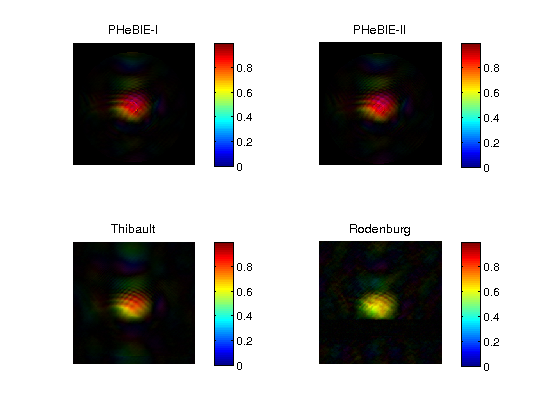}
     \caption{Reconstructed probes.}
  \end{subfigure}
  \begin{subfigure}[b]{\linewidth} 
     \centering
     \includegraphics[width=15cm, height=7cm]{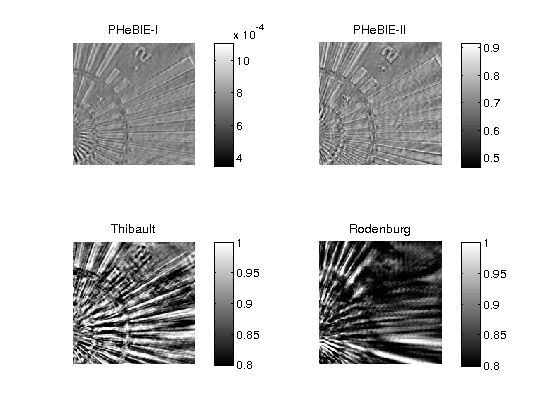}
     \caption{Reconstructed specimen amplitudes.}
  \end{subfigure}
  \begin{subfigure}[b]{\linewidth} 
     \centering
     \includegraphics[width=15cm, height=7cm]{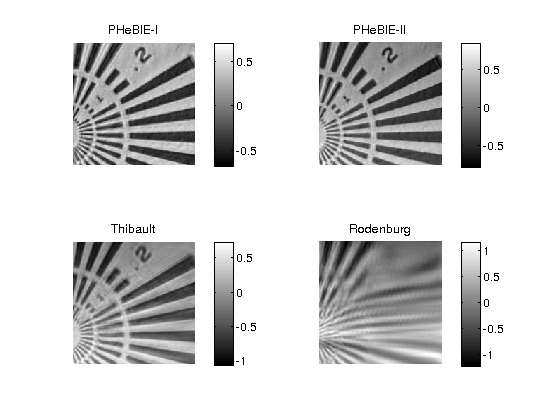}
     \caption{Reconstructed specimen phases.}
  \end{subfigure}
  \caption{Results for the experimental dataset for the four different algorithms.}\label{f:comp recon exp}
\end{figure}

\begin{figure}[h]
  \centering
  \begin{subfigure}[b]{\linewidth}
     \centering
     \includegraphics[width=0.9\linewidth]{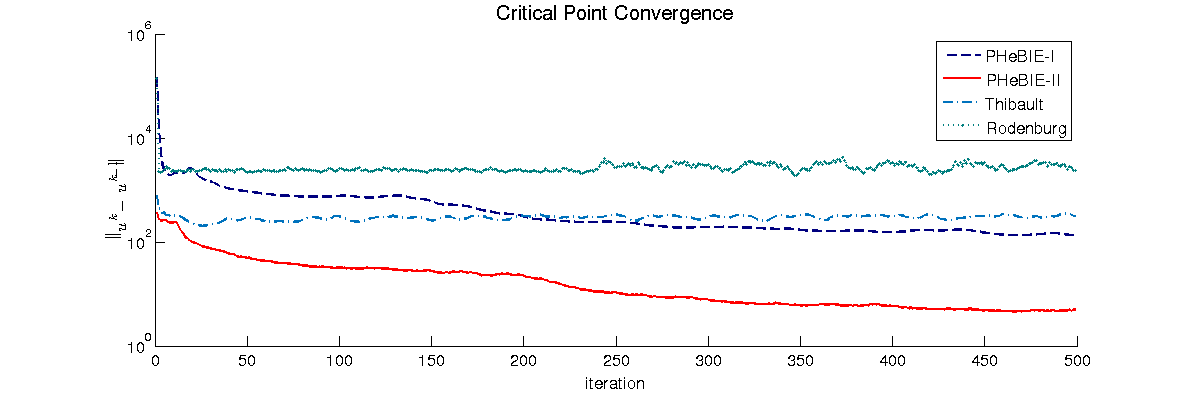}
     \caption{The norm of the differences between successive iterates.}
  \end{subfigure}
  
  \medskip
  
  \begin{subfigure}[b]{\linewidth}
     \centering
     \includegraphics[width=0.9\linewidth]{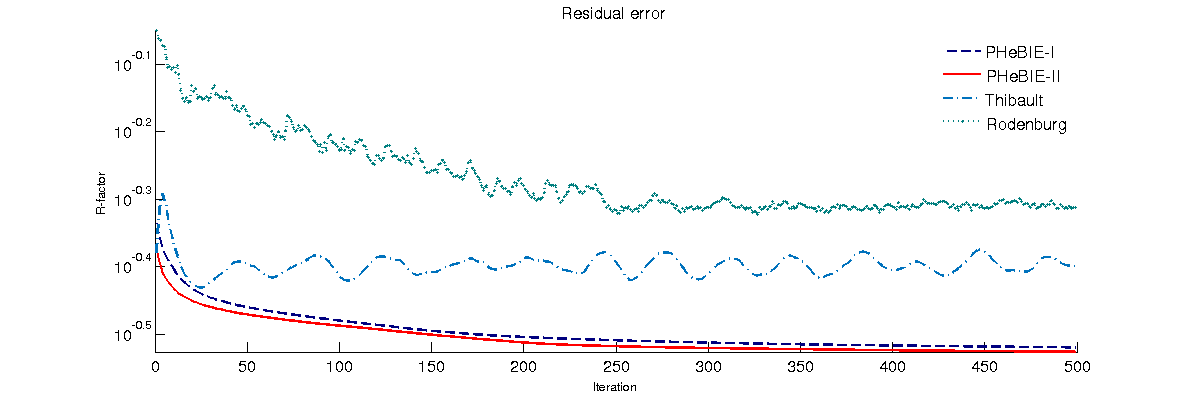}
     \caption{The $R$-factor of the iterates defined by \eqref{e:R-factor}.}
  \end{subfigure}
 \caption{Performance profiles for the four algorithms applied to experimental data.}\label{f:comp convergence exp}
\end{figure}

\section{Appendix} 

\subsection{Appendix A: Proof of Equations \eqref{e:nablax F ptychography splits}-\eqref{e:nablayi F ptychography modulus}}
	First, we compute the partial gradient of both functions. 
	\begin{eqnarray}
		\nabla_{x} F\left(x	, y , {\bf z}\right) & = & 2\sum_{j = 1}^{m} 
		\left[S_{j}\left(\cdot\right) \odot y\right]^{\ast}\left(S_{j}\left(x\right) \odot y 
		- z_{j}\right) = 2\sum_{j = 1}^{m} S_{j}^{\ast}\left(\overline{y} \odot 
		\left(S_{j}\left(x\right) \odot y - z_{j}\right)\right) \nonumber \\
		& = & 2\sum_{j = 1}^{m} \left[S_{j}^{\ast}\left(\overline{y} \odot y\right) \odot 
		x - S_{j}^{\ast}\left(\overline{y} \odot z_{j}\right) \right] 
		\label{PartialDerivativex}
	\end{eqnarray}
	and
	\begin{eqnarray}
		\nabla_{y} F\left(x	, y , {\bf z}\right) & = & 2\sum_{j = 1}^{m} 
		\left[S_{j}\left(x\right) \odot 
		\left(\cdot\right)\right]^{\ast}\left(S_{j}\left(x\right) \odot y 
		- z_{j}\right) = 2\sum_{j = 1}^{m} \overline{S_{j}\left(x\right)} \odot 
		\left(S_{j}\left(x\right) \odot y - z_{j}\right) \nonumber \\
		& = & 2\sum_{j = 1}^{m} \left[S_{j}\left(\overline{x} \odot x\right) \odot y - 
		S_{j}\left(\overline{x}\right) \odot z_{j} \right], 
		\label{PartialDerivativey}
	\end{eqnarray}
	where $S_{j}^{\ast}$, $j = 1 , 2 , \ldots , m$, denotes the adjoint transformation of 
	$S_{j}$ and $\overline{z}$ denote the element-wise complex conjugate of $z$. We remind 
	the reader that $S_{j}^{\ast} = S_{j}^{-1}$. We also used the following two facts:
	\begin{equation*}
		\left[S_{j}\left(\cdot\right) \odot y\right]^{\ast} = S_{j}^{\ast}\left(\overline{y} 
		\odot \left(\cdot\right)\right) \quad \text{and} \quad \left[S_{j}\left(x\right) 
		\odot \left(\cdot\right)\right]^{\ast} =  \overline{S_{j}\left(x\right)} \odot 
		\left(\cdot\right).
	\end{equation*}
	Using \eqref{PartialDerivativex} we obtain, for any $x , x' \in \comp^{n}$ that
	\begin{align}
		\nabla_{x} F\left(x	, y , \bz\right) - \nabla_{x} F\left(x' , y , \bz\right) & =  
		2\sum_{j = 1}^{m} \left[S_{j}^{\ast}\left(\overline{y} \odot y\right) \odot 
		x - S_{j}^{\ast}\left(\overline{y} \odot y\right) \odot x' \right] \nonumber \\
		& = 2\sum_{j = 1}^{m} S_{j}^{\ast}\left(\overline{y} \odot y\right) \odot 
		\left(x - x'\right) \nonumber \\
		& = 2\left(\sum_{j = 1}^{m} S_{j}^{\ast}\left(\overline{y} \odot y\right)\right) 
		\odot \left(x - x'\right), \label{PartialDerivativeDifferencex}
	\end{align}
	which means that
	\begin{equation*}
		\norm{\nabla_{x} F\left(x , y , \bz\right) - \nabla_{x} F\left(x' , y , 
		\bz\right)} \leq 2\norm{\sum_{j = 1}^{m} S_{j}^{\ast}\left(\overline{y} \odot 
		y\right)}_{\infty} \cdot \norm{x - x'},
	\end{equation*}
	the last inequality follows from the following fact
	\begin{equation*}
		\norm{u \odot v}^{2} = \sum_{j = 1}^{m} \left(u_{j}v_{j}\right)^{2} \leq \sum_{j = 
		1}^{m} \left(\left|u_{j*}\right| \cdot \left|v_{j}\right|\right)^{2} = u_{j*}^{2}
		\sum_{j = 1}^{m} v_{j}^{2} = \norm{u}_{\infty}^{2}\norm{v}^{2},
	\end{equation*}
	where $j*$ is the index of the largest entry in absolute value of $u$. This proves that 
	\begin{equation*}
		L_{x}\left(y , z\right) \leq 2\norm{\sum_{j = 1}^{m} S_{j}^{\ast}\left(\overline{y} 
		\odot y\right)}_{\infty}.
	\end{equation*}
  	On the other hand, choosing $x' = 0$ and $x = e_{i}$ (which is the $i$-th standard unit 
  	vector) and using \eqref{PartialDerivativeDifferencex} shows that
	\begin{equation*}
		\nabla_{x} F\left(x	, y , \bz\right) - \nabla_{x} F\left(x' , y , \bz\right) = 
		2\left(\sum_{j = 1}^{m} S_{j}^{\ast}\left(\overline{y} \odot y\right)\right) \odot 
		e_{i} = 2\left(\sum_{j = 1}^{m} S_{j}^{\ast}\left(\overline{y} \odot 
		y\right)\right)_{i},
	\end{equation*}
	where $(v)_{i}$ denotes the $i$-th component of the vector $v$. This means that we take 
	$i = j*$, the largest entry in absolute value of $\sum_{j = 1}^{m} 
	S_{j}^{\ast}\left(\overline{y} \odot y\right)$, then we obtain that
	\begin{equation*}
		\nabla_{x} F\left(x	, y , \bz\right) - \nabla_{x} F\left(x' , y , \bz\right) = 
		2\norm{\sum_{j = 1}^{m} S_{j}^{\ast}\left(\overline{y} 
		\odot y\right)}_{\infty}.
	\end{equation*}
	This shows that
	\begin{equation*}
		L_{x}\left(y , z\right) = 2\norm{\sum_{j = 1}^{m} S_{j}^{\ast}\left(\overline{y} 
		\odot y\right)}_{\infty}.
	\end{equation*}
	Similar arguments shows that
	\begin{equation*}
		L_{y}\left(x , z\right) = 2\norm{\sum_{j = 1}^{m} S_{j}\left(\overline{x} \odot 
		x\right)}_{\infty}.
	\end{equation*}	
	As a direct consequence of \eqref{PartialDerivativeDifferencex} we achieve
	\begin{equation*}
		L_{x_i}\left(y , z\right) = 2\left[\sum_{j = 1}^{m} S_{j}^{\ast}\left(\overline{y} 
		\odot y\right)\right]_i
	\end{equation*}
	and by similar argument
	\begin{equation*}
		L_{y_i}\left(x , z\right) = 2\left[{\sum_{j = 1}^{m} S_{j}\left(\overline{x} \odot 
		x\right)}\right]_i
	\end{equation*}
which are \eqref{e:nablaxi F ptychography modulus} and \eqref{e:nablayi F ptychography modulus} respectively.
\hfill$\Box$

    	\begin{remark}[Block partial Lipschitz constants]
     For more general variable blocks of the form considered in Section~\ref{SSec:Acceleration}, the corresponding formula for $L_{x_i}(y,z)$ (resp. $L_{y_i}(x,z)$) are given by taking twice largest entry in the block $X_i$ (resp. $Y_i$) from the summation. That is,
      $$L_{x_i}\left(y , z\right) = 2\left\|\left.\left(\sum_{j = 1}^{m} S_{j}^{\ast}\left(\overline{y} 
		\odot y\right)\right)\right|_{X_i}\right\|_\infty,\quad L_{y_i}\left(x , z\right) = 2\left\|\left.\left(\sum_{j = 1}^{m} S_{j}\left(\overline{x} \odot 
		x\right)\right)\right|_{Y_i}\right\|_\infty,$$
	where $|_{X_i}$ (res. $|_{Y_i}$) denotes the restriction to the block $X_i$ (resp. $Y_i$).
	
	From these formulae one immediately recovers \eqref{e:nablaxi F ptychography modulus} and \eqref{e:nablayi F ptychography modulus} as special cases.
	\end{remark}
	
\section*{Acknowledgments}
	RH and DRL were supported by DFG grant SFB755TPC2. SDS was supported by an Alexander von 
	Humboldt Postdoctoral Fellowship. MKT was supported by an Australian Postgraduate Award.
We would like to thank Robin Wilke and Tim Salditt of the Institute for X-ray Physics at the 
University of G\"ottingen for generously making their data available to us.  

% \bibliographystyle{plain}
% \bibliography{biblio}

\begin{thebibliography}{10}

\bibitem{R2008}
J.~M. Rodenburg.
\newblock Ptychography and related diffractive imaging methods.
\newblock {\em Adv. Imaging Electron Phys.}, 150:87--184, 2008.

\bibitem{AB2009}
H.~Attouch and J.~Bolte.
\newblock On the convergence of the proximal algorithm for nonsmooth functions
  involving analytic features.
\newblock {\em Math. Program.}, 116:5--16, 2009.

\bibitem{ABRS2010}
H.~Attouch, J.~Bolte, P.~Redont, and A.~Soubeyran.
\newblock Proximal alternating minimization and projection methods for
  nonconvex problems: an approach based on the {K}urdyka-{{\L}}ojasiewicz
  inequality.
\newblock {\em Math. Oper. Res.}, 35:438--457, 2010.

\bibitem{ABS2013}
H.~Attouch, J.~Bolte, and B.~F. Svaiter.
\newblock Convergence of descent methods for semi-algebraic and tame problems:
  proximal algorithms, forward-backward splitting, and regularized
  {G}auss-{S}eidel methods.
\newblock {\em Math. Program.}, 137:91--129, 2013.

\bibitem{BCL2004}
H.~H. Bauschke, P.~L. Combettes, and D.~R. Luke.
\newblock Finding best approximation pairs relative to two closed convex sets
  in {H}ilbert spaces.
\newblock {\em J. Approx. Theory}, 127:178--192, 2004.

\bibitem{BT2009}
A.~Beck and M.~Teboulle.
\newblock A fast iterative shrinkage-thresholding algorithm for linear inverse
  problems.
\newblock {\em SIAM J. Imaging Sci.}, 2:183--202, 2009.

\bibitem{BDL2006}
J.~Bolte, A.~Daniilidis, and A.~Lewis.
\newblock The {{\L}}ojasiewicz inequality for nonsmooth subanalytic functions
  with applications to subgradient dynamical systems.
\newblock {\em SIAM J. Optim.}, 17:1205--1223, 2006.

\bibitem{BDL2007}
J.~Bolte, A.~Daniilidis, A.~Lewis, and M.~Shiota.
\newblock Clarke subgradients of stratifiable functions.
\newblock {\em SIAM J. on Optimization}, 18:556--572, 2007.

\bibitem{BST2013}
J.~Bolte, S.~Sabach, and M.~Teboulle.
\newblock Proximal alternating linearized minimization for nonconvex and
  nonsmooth problems.
\newblock {\em Math. Program., Ser. A,}146:459--494 (2014).

\bibitem{CW2005}
P.~L. Combettes and V.~R. Wajs.
\newblock Signal recovery by proximal forward-backward splitting.
\newblock {\em Multiscale Model. Simul.}, 4:1168--1200, 2005.

\bibitem{DougRach56}
J.~Douglas and H.~H. Rachford.
\newblock On the numerical solution of heat conduction problems in two or three
  space variables.
\newblock {\em Trans. Amer. Math. Soc.}, 82:421--439, 1956.

\bibitem{guizar2008efficient}
M.~Guizar-Sicairos, S.T. Thurman, and J.R. Fienup.
\newblock Efficient subpixel image registration algorithms.
\newblock {\em Optics letters}, 33(2):156--158, 2008.

\bibitem{Hoppe}
R.~Hegerl and W.~Hoppe.
\newblock Dynamische theorie der kristallstrukturanalyse durch
  elektronenbeugung im inhomogenen primärstrahlwellenfeld.
\newblock {\em Ber. Bunsenges. Phys. Chem}, 74(11):1148--1154, 1970.

\bibitem{LionsMercier}
P.~L. Lions and B.~Mercier.
\newblock {Splitting Algorithms for the Sum of Two Nonlinear Operators}.
\newblock {\em SIAM Journal on Numerical Analysis}, 16(6):964--979, 1979.

\bibitem{L2008}
D.~R. Luke.
\newblock Finding best approximation pairs relative to a convex and a
  prox-regular set in a {H}ilbert space.
\newblock {\em SIAM J. Optim.}, 19:714--739, 2008.

\bibitem{LBL2002}
D.~R. Luke, J.~V. Burke, and R.~G. Lyon.
\newblock Optical wavefront reconstruction: theory and numerical methods.
\newblock {\em SIAM Rev.}, 44:169--224, 2002.

\bibitem{RAAR}
D.R. Luke.
\newblock Relaxed averaged alternating reflections for diffraction imaging.
\newblock {\em Inverse Problems}, 21(1):37, 2005.

\bibitem{MR2009}
A.~M. Maiden and J.~M. Rodenburg.
\newblock An improved ptychographical phase retrieval algorithm for diffractive
  imaging.
\newblock {\em Ultramicroscopy}, 109:1256--1262, 2009.

\bibitem{YangMarchesini}
J.~Qian, C.~Yang, A.~Schirotzek, F.~Maia, and S.~Marchesini.
\newblock Efficient algorithms for ptychographic phase retrieval.
\newblock {\em Contemporary Mathematics}, 2014.

\bibitem{RW1998-B}
R.~T. Rockafellar and R.~J. Wets.
\newblock {\em Variational Analysis}.
\newblock Grundlehren der mathematischen Wissenschaften. Springer-Verlag,
  Berlin, 1998.

\bibitem{RodenburgBates}
J.~M. Rodenburg and R.~H.~T. Bates.
\newblock The theory of super-resolution electron microscopy via
  wigner-distribution deconvolution.
\newblock {\em Philos. Trans. R. Soc. London Ser. A}, 339:521--553, 1992.

\bibitem{Rodenburg07}
J.~M. Rodenburg, A.~C. Hurst, A.~G. Cullis, B.~R. Dobson, F.~Pfeiffer, O.~Bunk,
  C.~David, K.~Jefimovs, and I.~Johnson.
\newblock Hard-x-ray lensless imaging of extended objects.
\newblock {\em Phys. Rev. Lett.}, 98:034801, 2007.

\bibitem{SECCMS}
Y.~Shechtman, Y.~C. Eldar, O.~Cohen, H.~N. Chapman, J.~Miao and M.~Segev.
\newblock Phase retrieval with application to optical imaging.
\newblock {\em IEEE Magazine}, 2014.

\bibitem{TDBMP2009}
P.~Thibault, M.~Dierolf, O.~Bunk, A.~Menzel, and F.~Pfeiffer.
\newblock Probe retrieval in ptychographic coherent diffractive imaging.
\newblock {\em Ultramicroscopy}, 109:338--343, 2009.

\bibitem{Wilke}
R.~N. Wilke, M.~Priebe, M.~Bartels, K.~Giewekemeyer, A.~Diaz, P.~Karvinen, and
  T.~Salditt.
\newblock Hard {X}-ray imaging of bacterial cells: nano-diffraction and
  ptychographic reconstruction.
\newblock {\em Optics Express}, 20(17):19232--19254, 2012.

\end{thebibliography}

\end{document}